\documentclass[reqno]{amsart}

\usepackage{verbatim}
\usepackage{graphicx}
\usepackage{amssymb}
\usepackage{esint}

\theoremstyle{plain}
\newtheorem{thm}{Theorem}[section]
\newtheorem{lemma}[thm]{Lemma}
\newtheorem{defn}[thm]{Definition}

\numberwithin{equation}{section}

\newtheorem{rem}[thm]{Remark}

\begin{document}
\title[Landau-Lifshitz-Gilbert equation]{Global Well-Posedness of the Landau-Lifshitz-Gilbert equation for initial data in Morrey space} 

\author[J. Lin]{Junyu Lin} 
\author[B. Lai]{Baishun Lai} 
\author[C. Wang]{Changyou Wang}
\address{Department of Mathematics\\ South China  University of Technology\\ Guangzhou,  Guangdong 510640, China}
\address{School of Mathematics and Information Science\\
Henan University\\ Kaifeng 475004, Henan, China}
\address{Department of Mathematics\\
University of Kentucky \\
Lexington, KY 40506}

\email{scjylin@scut.edu.cn,  laibaishun@gmail.com, cywang@ms.uky.edu} 
\date{\today}
\subjclass[2000]{}

\begin{abstract}
We establish the global well-posedness of the Landau-Lifshitz-Gilbert equation in $\mathbb R^n$ for any initial data ${\bf m}_0\in H^1_*(\mathbb R^n,\mathbb S^2)$
whose gradient belongs to the Morrey space $M^{2,2}(\mathbb R^n)$ with small norm $\displaystyle\|\nabla {\bf m}_0\|_{M^{2,2}(\mathbb R^n)}$. The method is based on
priori estimates of a dissipative Schr\"odinger equation of Ginzburg-Landau types obtained from the Landau-Lifshitz-Gilbert equation by
the moving frame technique.
\end{abstract}
\maketitle 

\section{Introduction}
\setcounter{equation}{0}
\setcounter{thm}{0}


The Landau-Lifshitz-Gilbert equation, originally introduced by Landau and Lifshitz \cite{LL} \cite{G}, serves as the basic
evolution equation for the spin fields in the continuum theory of ferromagnetism.
Let ${\bf m}:\mathbb R^n\times(0,+\infty)\rightarrow \mathbb S^2$ denote the spin director field
with values in the unit sphere $\mathbb S^2\subset\mathbb R^3$.
The  Landau-Lifshitz-Gilbert equation (LLG) is given by
\begin{equation}
\label{1.1}
\begin{cases}
\partial_t {\bf m}=-{\bf m}\times\triangle {\bf m}-\lambda {\bf m}\times ({\bf m}\times\triangle {\bf m}) & \ \mbox{in}\ \mathbb R^n\times(0,+\infty), \\
\ \ {\bf m}={\bf m}_0 & \ \mbox{in}\ \mathbb R^n\times\{0\},
\end{cases}
\end{equation}
where $\lambda>0$ is a Gilbert damping parameter, $\times$ is the cross product in $\mathbb R^3$, and ${\bf m}_0:\mathbb R^n\to\mathbb S^2$
is an initial data. Note that $-{\bf m}\times {\bf m}\times \Delta {\bf m}=\Delta {\bf m}+|\nabla{\bf m}|^2{\bf m}$ is the tension field of ${\bf m}:\mathbb R^n\to\mathbb S^2$.
Hence LLG (\ref {1.1}) is a hybrid of the heat flow and the Schr\"odinger flow of harmonic maps to $\mathbb S^2$ and is an evolution equation
of parabolic types. A Liapunov functional of (\ref{1.1}) is given by the Dirichlet energy
$$E({\bf m})=\frac12\int_{\mathbb R^n} |\nabla{\bf m}|^2\,dx.$$

Since LLG (\ref{1.1}) is energy critical when $n=2$,  the global existence of weak solutions with at most finitely many singularities
and the uniqueness among energy non-increasing solutions has been established
by \cite{GH} and \cite{Harp1, Harp2}, which is based on arguments similar to Struwe's approach to the heat flow of harmonic maps \cite{Struwe1}.
We would like to point out that the existence of finite time singularity of (\ref{1.1}) in dimension $n=2$ has not been shown yet.

When $n\ge 3$, LLG (\ref{1.1}) is super-critical with respect to the Dirichlet energy. Although one can show the global existence of  weak solutions,
the question of regularity and uniqueness of weak solutions becomes a delicate issue.  In dimensions $n=3, 4$, partial regularity for suitable weak
solutions to (\ref{1.1}) has been obtained by Moser \cite{Moser}, while the existence of partially regular weak solutions to (\ref{1.1}) has been
established by Melcher \cite{Melcher0} and Wang \cite{Wang}. However, since the arguments in \cite{Moser, Melcher0, Wang} are perturbation methods based on energy monotonicity
inequalities and elliptic estimates on generic time slices, they fail to work for $n\ge 5$. It may be worthy to point out that neither
the Bochner formula nor Struwe's parabolic energy monotonicity formula for the heat flow of harmonic maps (or approximated harmonic maps) \cite{Struwe2,
Chen-Struwe} seems to be available  for (\ref{1.1}). We would also like to point out that the existence of finite time singularity of LLG (\ref{1.1}) in dimensions
$n=3,4$ has been shown by \cite{Ding-Wang}.

Very recently, Melcher \cite{Melcher} investigated the global well-posedness of LLG (\ref{1.1}) in dimensions $n\ge 3$ for initial data ${\bf m}_0$ in the scaling
invariant homogeneous $\dot{W}^{1,n}(\mathbb R^n)$ space with small $\dot{W}^{1,n}$-norm, see also Seo \cite{Seo} for related works. The idea in \cite{Melcher} involves a transformation of
LLG (\ref{1.1}) by the technique of moving frame to a dissipative Schr\"odinger equation of complex Ginzburg-Landau type nonlinearities.
It is natural to establish the global well-posedness of (\ref{1.1}) for initial data ${\bf m}_0$ in larger classes of function spaces.
Indeed,  we are able to prove the global well-posedness of (\ref{1.1}) for initial data whose gradient is in certain Morrey spaces with small Morry norm.

In order to state our main results, we first recall

\begin{defn}\label{de:1.1} For $1\leq p<+\infty,$ $0\leq q\leq n,$ the Morrey space $M^{p,q}(\mathbb R^n)$
 is defined by
\begin{equation}
 M^{p,q}(\mathbb R^n)
:=\Big\{f\in L^p_{{\rm{loc}}}(\mathbb R^n):\ \big\|f\big\|_{M^{p,q}(\mathbb R^n)}^p=\sup\limits_{x\in \mathbb R^n,0<r<+\infty}r^{q-n}\int_{B_r(x)}|f(y)|^p
\,dy<+\infty\Big\}.
\end{equation}
\end{defn}

Observe that $M^{p,n}(\mathbb R^n)=L^p(\mathbb R^n)$, and from the point of view of scalings, $M^{p,p}(\mathbb R^n)$ behaves like $L^n(\mathbb R^n)$.

Now we are ready to state our main theorem.
\begin{thm}\label{th:1.1} For $\lambda>0$ and $n\ge 2$, there exist $\epsilon_0>0$ and $c_0>0$ depending only on $n$ and $\lambda$ with the following properties:
\begin{itemize}
\item[(i)]{\rm{(global existence)}} If ${\bf m}_0:\mathbb R^n\rightarrow \mathbb S^2$
satisfies ${\bf m}_0-{\bf m}_\infty\in L^2(\mathbb R^n)$ for some ${\bf m}_\infty\in \mathbb S^2$,  with $\nabla {\bf m}_0\in M^{2,2}(\mathbb R^n)$ satisfying
\begin{equation}\label{1.2}\big\|\nabla {\bf m}_0\big\|_{M^{2,2}(\mathbb R^n)}\leq\epsilon_0,
\end{equation}
then there exists a global solution ${\bf m}:\mathbb R^n\times [0,+\infty)\to\mathbb S^2$ to LLG (\ref{1.1}) such that
\begin{equation}\label{x-bound}
\sup_{t\ge 0}\big\|\nabla {\bf m}\big\|_{M^{2,2}(\mathbb R^n)}\le c_0 \big\|\nabla {\bf m}_0\big\|_{M^{2,2}(\mathbb R^n)}.
\end{equation}
Moreover, ${\bf m}\in C^\infty(\mathbb R^n\times (0,+\infty),\mathbb S^2)$ and satisfies
\begin{equation}
\sup_{t>0}\ t^{\frac{k}2}\Big\|\nabla^k{\bf m}(t)\Big\|_{L^\infty(\mathbb R^n)}\le c(n,k)\big\|\nabla{\bf m}_0\big\|_{M^{2,2}(\mathbb R^n)},
\ \forall k\ge 1, \ \label{smoothness}
\end{equation}
and ${\bf m}(t)\rightarrow {\bf m}_0$ in $H^1_{\rm{loc}}(\mathbb R^n, \mathbb S^2)$ as $t\rightarrow 0^+$.
\item[(ii)]{\rm{(uniqueness)}} If, in addition, $\nabla {\bf m}_0\in L^2(\mathbb R^n)$, then the global solution ${\bf m}$ also satisfies
that ${\bf m}\in C([0,+\infty), H^1_*(\mathbb R^n))\ {\rm{with}}\ \partial_t {\bf m}\in L^2([0,+\infty), L^2(\mathbb R^n)),$ and
the energy inequality:
\begin{equation}\label{energy_ineq0}
E({\bf m}(t))+\frac{\lambda}{1+\lambda^2}\int_0^t\int_{\mathbb R^n}|\partial_t{\bf m}|^2\,dxdt\le E({\bf m}_0), \ \forall\ t>0.
\end{equation}
Moreover, ${\bf m}$ is unique in its own class.
\end{itemize}
\end{thm}

Before we proceed with the presentation, we would like to make a few remarks concerning Theorem \ref{th:1.1}.
\begin{rem}{\rm For LLG (\ref{1.1}), we have the following comments.
\begin{itemize}
\item[(i)] Since H\"older's inequality implies that $L^n(\mathbb R^n)\subset M^{2,2}(\mathbb R^n)$, Theorem \ref{th:1.1} improves the main result of Melcher \cite{Melcher}
and Seo \cite{Seo}.
\item[(ii)] Without the smallness condition (\ref{1.2}), the short time smooth solution of (\ref{1.1}) can develop singularity at finite time for $n\ge 3$. In fact, the initial data
${\bf m}_0$ in the example of finite time singularity constructed by
Ding-Wang \cite{Ding-Wang} has small Dirichlet energy $E({\bf m}_0)$ and finite $\|\nabla{\bf m}_0\|_{M^{2,2}}$.
\item[(iii)] Motivated by the well-posedness result  on the heat flow of harmonic maps by \cite{Wang1}, it seems reasonable to conjecture that LLG (\ref{1.1})
is globally well-posed for any initial data ${\bf m}_0:\mathbb R^n\to \mathbb S^2 $ having small BMO norm $[{\bf m}_0]_{\rm{BMO}(\mathbb R^n)}$.
\end{itemize}
}
\end{rem}

We briefly discuss some of the ideas in the proof. The first crucial step is to utilize a canonical choice of coordinates on the tangent bundle $T\mathbb S^2$, called
moving frames, to convert (\ref{1.1}) into a covariant version of LLG, which is a (nonlocal) semilinear complex valued Schr\"odinger equation with cubic nonlinearity.
Then, by choosing a Coulomb gauge frame, one can get the desired control of nonlocal terms. Finally, using estimates of the dissipative Schr\"odinger semigroup
$S=S(t)$, generated by $(\lambda-i)\Delta$, between Morrey spaces, one can get the desired priori bounds for the short time approximate solutions to (\ref{1.1})
under the smallness assumption (\ref{1.2}). It seems that the estimates of $S(t)$ between Morrey spaces may have its own interests, with potential applications to
other types of equations.

The paper is written as follows. In section 2, we will establish some basic estimates of $S(t)$ between Morrey spaces.
In section 3, we review the derivation of covariant complex Ginzburg-Landau
type equation.  In section 4, we derive all the needed nonlinear
estimates of $S(t)$ between Morrey spaces. In section 5, we will prove Theorem \ref{th:1.1}.

\setcounter{section}{1}
\setcounter{equation}{0}
\section{Linear estimates for Schr\"odinger semigroup in Morrey spaces}

Throughout this paper, let $S(t)$ denote the semigroup generated by the dissipative Schr\"{o}dinger operator $(\lambda-i)\triangle$.
In this section, we will establish some basic estimates of $S(t)f=S_t*f$ in Morrey spaces when $f\in M^{p,q}(\mathbb R^n)$.

Recall that the Fourier transform of the associated kernel $S_t$ of $S(t)$ is given by
$$\hat{S}_t(\xi)=e^{(i-\lambda)|\xi|^2 t}, \ \xi\in\mathbb R^n, $$
so that the kernel $S_t$ can be written as
$$S_t(x)=t^{-\frac{n}2}S(\frac{x}{\sqrt{t}}),\ x\in\mathbb R^n,$$
where
$$S(y)=\int_{\mathbb R^n}e^{(i-\lambda)|y|^2}e^{iy\cdot\eta}\,d\eta, \ y\in\mathbb R^n,$$
is a radial Schwartz function in $\mathbb R^n$, i.e., $S(y)=S(|y|)$ for $y\in\mathbb R^n$.

For $f\in L^1_{\rm{loc}}(\mathbb R^n)$, the Hardy-Littlewood maximal function of $f$
is defined by (see, e.g. \cite{Stein})
$${\rm{M}}_{{\rm{HL}}}(f)(x):=\sup\limits_{r>0}\frac{1}{|B_r(x)|}\int_{B_r(x)}|f(y)|\,dy, \ x\in\mathbb R^n.$$
We will use $A\lesssim B$ to denote $A\leq CB$ for some universal positive constant $C$.
Now we have
\begin{lemma}\label{le:2.1} For $1<p<+\infty$ and $0\le q\le n$, if $f\in M^{p,q}(\mathbb R^n),$ then
$S(t)f\in M^{p(n+1),q}(\mathbb R^n)$,  and
\begin{equation}
\label{2.1}\Big\|S(t)f\Big\|_{M^{p(n+1),q}(\mathbb R^n)}\lesssim t^{-\frac{nq}{2p(n+1)}}\big\| f\big\|_{M^{p,q}(\mathbb R^n)}.
\end{equation}
\end{lemma}
\begin{proof}  It follows from the definition of $S(t)f$ that for any $x\in\mathbb R^n$,
\begin{eqnarray}
\nonumber (S(t)f)(x)&=&(S_t\ast f)(x)=\int_{\mathbb R^n}t^{-\frac{n}2}S\big(\frac{x-y}{\sqrt{t}}\big)f(y)\,dy\\
\nonumber&=& \Big(\int_{B_{\sqrt{t}\epsilon}(x)}+\int_{\mathbb R^n\setminus B_{\sqrt{t}\epsilon}(x)}\Big)t^{-\frac{n}2}S(\frac{x-y}{\sqrt{t}})f(y)\,dy\\
\nonumber&=& I(x)+II(x).
\end{eqnarray}
For $I(x),$ we have
\begin{equation}
\label{2.2}\big|I(x)\big|\leq\epsilon^n\big(\sqrt{t}\epsilon\big)^{-n}\int_{B_{\sqrt{t}\epsilon}(x)}|f(y)|\,dy\lesssim\epsilon^n{\rm{M}}_{\rm{HL}}(f)(x).
\end{equation}
To estimate $II(x)$, we proceed as follows.
\begin{eqnarray*}
 &&|II(x)|\leq\sum\limits_{k=1}^{\infty}t^{-\frac{n}2}\int_{B_{(k+1)\sqrt{t}\epsilon}(x)\backslash B_{k\sqrt{t}\epsilon}(x)}|S(\frac{x-y}{\sqrt{t}})||f(y)|\,dy\\
&&\leq t^{-\frac{q}{2p}}\Big[\sum\limits_{k=1}^{\infty}(k\epsilon)^{n-\frac{q}{p}}S^*(k\epsilon)\Big]\Big[\sup\limits_{r>0}r^{\frac{q}{p}-n}\int_{B_r(x)}|f(y)|\,dy\Big],
\end{eqnarray*}
where
$$S^*(k\epsilon):=\max\Big\{|S(y)|: \ y\in B_{(k+1)\epsilon}\setminus B_{k\epsilon}\Big\}.$$
Since $S$ is a Schwartz function, one sees that
\begin{eqnarray*}
\sum\limits_{k=1}^{\infty}(k\epsilon)^{n-\frac{q}{p}}S^*(k\epsilon)
\lesssim\frac{1}{\epsilon}\int_0^\infty t^{n-\frac{q}{p}}|S(t)|\,dt\lesssim\frac{1}{\epsilon}.
\end{eqnarray*}
This implies
\begin{equation}\label{2.3}
|II(x)|\lesssim t^{-\frac{q}{2p}}\epsilon^{-1}\big\|f\big\|_{M^{p,q}(\mathbb R^n)},
\end{equation}
where we have  used in the last step the inequality:
$$\Big\|f\Big\|_{M^{1,\frac{q}{p}}(\mathbb R^n)}\leq\Big\|f\Big\|_{M^{p,q}(\mathbb R^n)}.$$

Putting (\ref{2.2}) and (\ref{2.3}) together yields
\begin{eqnarray*}
\big|(S(t)f)(x)\big|\lesssim\epsilon^n{\rm{M}}_{{\rm{HL}}}(f)(x)+t^{-\frac{q}{2p}}\epsilon^{-1}\big\|f\big\|_{M^{p,q}(\mathbb R^n)}.
\end{eqnarray*}
Choose $\epsilon>0$ such that
$$\epsilon^n{\rm{M}}_{\rm{HL}}(f)(x)=t^{-\frac{q}{2p}}\epsilon^{-1}\big\| f\big\|_{M^{p,q}(\mathbb R^n)},$$
or,
$$\epsilon=\Big\{\frac{\big\|f \big\|_{M^{p,q}(\mathbb R^n)}}{t^{\frac{q}{2p}}{\rm{M}}_{\rm{HL}}(f)(x)}\Big\}^{\frac{1}{n+1}}.$$
Then we have
\begin{eqnarray*}
\big|S(t) f(x)\big|\lesssim t^{-\frac{nq}{2p(n+1)}}\big\| f\big\|_{M^{p,q}(\mathbb R^n)}^{\frac{n}{n+1}}\Big({\rm{M}}_{\rm{HL}}(f)(x)\Big)^{\frac{1}{n+1}},
\ x\in\mathbb R^n.
\end{eqnarray*}
Since ${\rm{M}}_{\rm{HL}}: M^{p,q}(\mathbb R^n)\rightarrow M^{p,q}(\mathbb R^n)$ is a bounded linear operator (see \cite{Adams} and \cite{HW1}),
we have that $S(t)f\in M^{p(n+1),q}(\mathbb R^n)$, and
\begin{eqnarray*}
\Big\|S(t)f\Big\|_{M^{p(n+1),q}(\mathbb R^n)}&&\lesssim t^{-\frac{nq}{2p(n+1)}}\big\|f\big\|_{M^{p,q}(\mathbb R^n)}^{\frac{n}{n+1}}
\big\|f\big\|_{M^{p,q}(\mathbb R^n)}^{\frac{1}{n+1}}\\
&&\lesssim t^{-\frac{nq}{2p(n+1)}}\big\| f\big\|_{M^{p,q}(\mathbb R^n)}.
\end{eqnarray*}
This completes the proof of Lemma \ref{le:2.1}.
\end{proof}

\begin{lemma}\label{le:2.2} For $1<p<+\infty$ and $0\le q\le n$, if $f\in M^{p,q}(\mathbb R^n),$ then $S(t)f\in M^{p,q}(\mathbb R^n)$ and
\begin{equation}
\label{2.4}\Big\|S(t)f\Big\|_{M^{p,q}(\mathbb R^n)}\leq \big\| f \big\|_{M^{p,q}(\mathbb R^n)}.
\end{equation}
\end{lemma}
\begin{proof} This follows from the Young inequality for convolution operators. In fact,
 $$(S(t) f)(z)=\int_{\mathbb R^n}S_t(z-y)f(y)\,dy=\int_{\mathbb R^n}S_t(y)f(z-y)\,dy,$$ we have
\begin{eqnarray*}
\Big\|S(t)f\Big\|_{L^p(B_R(x))}^p\le \big(\int_{\mathbb R^n}|S_t(y)|\,dy\big)^p\big\|f\big\|_{L^p(B_R(x))}^p\lesssim \int_{B_R(x)}|f(y)|^p\,dy,
\end{eqnarray*}
so that
\begin{eqnarray}
R^{q-n}\int_{B_R(x)}\big|S(t) f(z)\big|^p\,dz\lesssim R^{q-n}\int_{B_R(x)}|f(z)|^p\,dz. \label{2.4.0}
\end{eqnarray}
Taking supremum over all $B_R(x)\subset\mathbb R^n$ in (\ref{2.4.0})  yields (\ref{2.4}).
Here we have used the fact that
$$\int_{\mathbb R^n}\big|S_t(y)\big|\,dy=\int_{\mathbb R^n}\big|S(y)\big|\,dy\lesssim 1.$$
This completes the proof.
\end{proof}

\begin{lemma}\label{le:2.3} {\rm{(Interpolation\ lemma)}} For $1<p<+\infty$ and $0\le q\le n$,
if $f\in M^{p,q}(\mathbb R^n),$ then $S(t)f\in M^{\widetilde{p},q}(\mathbb R^n)$ for any $p\leq\widetilde{p}\leq p(n+1),$ and \begin{equation}
\label{2.5}\Big\|S(t)f\Big\|_{M^{\widetilde{p},q}(\mathbb R^n)}\lesssim t^{-\frac{q}2(\frac{1}{p}-\frac{1}{\widetilde{p}})}\big\| f \big\|_{M^{p,q}(\mathbb R^n)}.
\end{equation}
\end{lemma}
\begin{proof} For any $\widetilde{p}\in[p,p(n+1)],$ write $$\frac{1}{\widetilde{p}}=\frac{\theta}{p}+\frac{1-\theta}{p(n+1)}.$$
Then by H\"older's inequality, Lemma \ref{2.1} and Lemma \ref{2.2}, we have
\begin{eqnarray*}
\Big\|S(t)f\Big\|_{M^{\widetilde{p},q}(\mathbb R^n)}&\lesssim&\Big\|S(t)f\Big\|_{M^{p,q}(\mathbb R^n)}^\theta\Big\|S(t) f\Big\|_{M^{p(n+1),q}(\mathbb R^n)}^{1-\theta}\\
&\lesssim & t^{-\frac{nq(1-\theta)}{2p(n+1)}}\big\|f\big\|_{M^{p,q}(\mathbb R^n)}^{\theta}\big\|f\big\|_{M^{p,q}(\mathbb R^n)}^{1-\theta}\\
&\lesssim & t^{-\frac{nq(1-\theta)}{2p(n+1)}}\big\| f\big\|_{M^{p,q}(\mathbb R^n)}.
\end{eqnarray*}
Since $$1-\theta=\frac{\frac{1}{p}-\frac{1}{\widetilde{p}}}{\frac{n}{p(n+1)}},$$
we have $\frac{n(1-\theta)}{p(n+1)}=\frac{1}{p}-\frac{1}{\widetilde{p}}$ and hence
$$\Big\|S(t) f\Big\|_{M^{\widetilde{p},q}(\mathbb R^n)}\leq t^{-\frac{q}2(\frac{1}{p}-\frac{1}{\widetilde{p}})}\big\| f \big\|_{M^{p,q}(\mathbb R^n)}.$$
This proves Lemma \ref{le:2.3}.
\end{proof}

We also need to estimate $\nabla(S(t) f)$ in Morrey spaces. More precisely, we have

\begin{lemma}\label{le:5.1} If $f\in M^{p,q}(\mathbb R^n)$  then $\nabla(S(t)f)\in M^{p(n+1),q}(\mathbb R^n)$, and
\begin{equation}
\label{5.1}\Big\|\nabla(S(t)f)\Big\|_{M^{p(n+1),q}(\mathbb R^n)}\lesssim t^{-\frac12-\frac{qn}{2p(n+1)}}\Big\|f\Big\|_{M^{p,q}(\mathbb R^n)}.
\end{equation}
\end{lemma}
\begin{proof}
Since $$\nabla S_t(x)=\nabla\big[(t^{-\frac{n}2}S(\frac{x}{\sqrt{t}})\big]=t^{-\frac{n+1}2}\nabla S\big(\frac{x}{\sqrt{t}}\big),$$
and $\nabla S$ is also a Schwartz function on $\mathbb R^n$,
we can apply the same argument as in Lemma \ref{le:2.1} to get that
\begin{eqnarray*}
t^\frac12\Big|\nabla(S(t)f)(x)\Big|\lesssim t^{-\frac{qn}{2p(n+1)}}\Big\|f\Big\|_{M^{p,q}(\mathbb R^n)}\Big({\rm{M}}_{\rm{HL}}(f)(x)\Big)^{\frac{1}{n+1}}
\end{eqnarray*}
so that
\begin{eqnarray*}
t^{\frac12}\Big\|\nabla(S(t) f)\Big\|_{M^{p(n+1),q}(\mathbb R^n)}\lesssim t^{-\frac{qn}{2p(n+1)}}\Big\|f\Big\|_{M^{p,q}(\mathbb R^n)}.
\end{eqnarray*}
This yields (\ref{5.1}).
\end{proof}

\begin{lemma}\label{le:5.2} If $f\in M^{p,q}(\mathbb R^n)$  then $\nabla(S(t) f)\in M^{p,q}(\mathbb R^n)$ and
\begin{equation}
\label{5.2}\Big\|\nabla(S(t)f)\Big\|_{M^{p,q}(\mathbb R^n)}\lesssim t^{-\frac12}\Big\|f\Big\|_{M^{p,q}(\mathbb R^n)}.
\end{equation}
\end{lemma}
\begin{proof} It is similar to Lemma \ref{2.2}. Since
$$\nabla (S(t)f)(z)=t^{-\frac12}\int_{\mathbb R^n}t^{-\frac{n}2}(\nabla S)\big(\frac{y}{\sqrt{t}}\big)f(z-y)\,dy,$$
the Young inequality implies that
$$\Big\|\nabla(S(t)f)\Big\|_{L^p(B_R(x))}^p
\le t^{-\frac{p}2}\big(\int_{\mathbb R^n} |\nabla S(y)|\,dy\big)^p \Big\|f\Big\|_{L^p(B_R(x))}^p
\lesssim t^{-\frac{p}2}\int_{B_R(x)}|f(y)|^p\,dy,$$
which, after taking supremum over all $B_R(x)\subset\mathbb R^n$, yields (\ref{5.2}).
Here we have used the fact that
$$\int_{\mathbb R^n}|\nabla S(y)|\,dy\lesssim 1.$$
This completes the proof.
\end{proof}
\begin{lemma}\label{le:5.3} {\rm{(Interpolation\ lemma)}} If $f\in M^{p,q}(\mathbb R^n),$ then
$\nabla(S(t)f)\in M^{\widetilde{p},q}(\mathbb R^n)$ for any $p\leq\widetilde{p}\leq p(n+1),$
and \begin{equation}
\label{5.3}\Big\|\nabla(S(t)f)\Big\|_{M^{\widetilde{p},q}(\mathbb R^n)}\leq t^{-\frac12-\frac{q}2(\frac{1}{p}-\frac{1}{\widetilde{p}})}\Big\|f\Big\|_{M^{p,q}(\mathbb R^n)}.
\end{equation}
\end{lemma}
\begin{proof} By Lemma 2.5, Lemma 2.6, and H\"older's inequality, (\ref{5.3}) can be proved exactly as in Lemma \ref{2.3}. We omit the detail here.
\end{proof}

\setcounter{section}{2}
\setcounter{equation}{0}
\section{Covariant Landau-Lifshitz-Gilbert equation}

The moving frame technique was first successfully applied to the study of harmonic maps by H\'elein \cite{Helein}, and subsequently used in the
study of wave maps (see \cite{FMS} and \cite{Shatah-Struwe}) and Schr\"odinger maps (see \cite{BIK}, \cite{BIKT} and \cite{NSU}). The basic idea is to use orthonormal frames
on the tangent bundle of the target manifold under the Coulomb gauge to rewrite the equation. It turns out that such a technique has been used by
Melcher \cite{Melcher} to derive another version of LLG (\ref{1.1}), which was called as the covariant LLG.
Now we briefly review such a construction and refer the readers to \cite{Melcher} for more details.

\subsection{Moving frames} Fix a point ${\bf m}_\infty\in \mathbb S^2$, we define, for $\sigma\in \mathbb Z_+$,  the homogeneous Sobolev space
$$H^\sigma_*(\mathbb R^n, \mathbb S^2)=\big\{{\bf m}:\mathbb R^n\to\mathbb S^2: \ {\bf m}-{\bf m}_\infty\in H^\sigma(\mathbb R^n, \mathbb R^3)\big\},$$
and
$$H^\infty_*(\mathbb R^n,\mathbb S^2)=\bigcap_{\sigma\in \mathbb Z_+}H^\sigma_*(\mathbb R^n,\mathbb S^2).$$

For $0<T<+\infty$, we consider
$${\bf m}\in C^0\big([0,T], H^\infty_*(\mathbb R^n, \mathbb S^2)\big) \ {\rm{with}}\ \partial_t {\bf m}\in C^0\big([0,T], H^\infty_*(\mathbb R^n, \mathbb R^3)\big).$$
Since the pull-back tangent bundle ${\bf m}^{-1}T\mathbb S^2$ is a trivial bundle on $\mathbb R^n\times [0,T]$, there exists a global orthonormal frame on
${\bf m}^{-1}T\mathbb S^2$: there are smooth tangent vector fields $X,Y\in C^\infty(\mathbb R^n\times [0,T], T_{\bf m}\mathbb S^2)
$ along ${\bf m}$ such that
$$|X|=|Y|=1, \ \langle X, Y\rangle =0, \ {\rm{and}}\  X\times Y={\bf m}.$$

We use $\alpha,\beta\in \{0,1,\cdots n\}$ to denote the space-time components, where $\alpha=0$ is the time index so that $\partial_0=\partial_t$.
Let $\langle\cdot,\cdot\rangle$ be the inner product on $\mathbb R^3$.
Denote the associated connection coefficient
$${\bf a}=(a_0,a_1,\cdots,a_n)=\big(\langle\partial_0 X,Y\rangle, \langle\partial_1 X,Y\rangle, \cdots, \langle\partial_nX,Y\rangle\big)
\in C^\infty(\mathbb R^n\times [0,T], \mathbb R^{n+1}).$$
This gives the covariant derivative $D_\alpha=\partial_\alpha+ia_\alpha$, $0\le\alpha\le n$.
Write
\begin{equation}
\label{complex-d}
{\bf u}:=(u_0, u)=(u_0, (u_1,\cdots, u_n))=\big(\langle\partial_\alpha {\bf m},X\rangle+i\langle\partial_\alpha{\bf m},Y\rangle\big)_{0\le\alpha\le n}:\mathbb R^n\times [0,T]\to\mathbb C^{n+1}
\end{equation}
for the coefficient of space-time gradient of ${\bf m}$ in terms of $X+iY$. Then we have
$$\partial_\alpha {\bf m}={\rm{Re}}(u_\alpha)X+{\rm{Im}}(u_\alpha)Y.$$
By the relations
$$\partial_\alpha X=-{\rm{Re}}(u_\alpha){\bf m}+a_\alpha Y,\ \ \ \partial_\alpha Y=-{\rm{Im}}(u_\alpha){\bf m}-a_\alpha X,$$
we have the zero torsion identity:
\begin{equation}
D_\alpha u_\beta=D_\beta u_\alpha,
\label{torsion}
\end{equation}
and the curvature identity:
\begin{equation}\label{curvature}
R_{\alpha\beta}:=[D_\alpha,D_\beta]=i{\rm{Im}}(u_\alpha u_\beta).
\end{equation}

\subsection{Covariant LLG}
Direct calculations imply
\begin{equation}\label{tension}
\Delta{\bf m}+|\nabla{\bf m}|^2 {\bf m}=\sum_{k=1}^n(\partial_k{\rm{Re}}(u_k)-a_k{\rm{Im}}(u_k)) X+\sum_{k=1}^n
(\partial_k{\rm{Im}}(u_k)+a_k{\rm{Re}}(u_k)) Y\ \ {\rm{in}}\ \ \mathbb R^n\times [0,T].
\end{equation}

By direct calculations, we will have the following result (see \cite{Melcher} Proposition 2).
\begin{thm}\label{covariant_LLG} For ${\bf m}_0\in H^\infty_*(\mathbb R^n,\mathbb S^2)$,
let ${\bf m}\in C^0\big([0,T], H^\infty_*(\mathbb R^n, \mathbb R^3)\big)$, with $\partial_t{\bf m}\in
C^0\big([0,T], H^\infty_*(\mathbb R^n, \mathbb R^3)\big)$, solve the LLG (\ref{1.1}) with initial data ${\bf m}_0$,
for some $T>0$.
Then
$$({\bf u}, {\bf a})\in C^0([0,T], H^\infty(\mathbb R^n, \mathbb C^{n+1}\times \mathbb R^{n+1}))$$
solves the covariant LLG:
\begin{equation}
\label{3.1}
\begin{cases}
u_0=(\lambda-i)\sum\limits_{k=1}^{n}D_ku_k, \\
D_\alpha u_\beta=D_\beta u_\alpha, \ 0\le\alpha,\beta\le n, \\
\partial_\alpha a_\beta-\partial_\beta a_\alpha={\rm{Im}}(u_\alpha \overline{u_\beta}),\ 0\le\alpha,\beta\le n.
\end{cases}
\end{equation}
Moreover, $u=(u_1,\cdots,u_n)$ solves the covariant complex Ginzburg-Landau equation:
\begin{equation}
\label{3.2}
\begin{cases}
D_0u_l=(\lambda-i)\sum\limits_{k=1}^{n}(D_k D_ku_l+R_{lk}u_k),\ \ \ 1\leq l\leq n,\\
u(0)=\langle\nabla {\bf m}_0,X\rangle+i\langle\nabla {\bf m}_0,Y\rangle.
\end{cases}
\end{equation}
\end{thm}

\subsection{Gauge invariance and Coulomb gauges}
Since (\ref{3.1}) is invariant under the gauge transformation:
\begin{equation}
{\bf u}\mapsto \widetilde {\bf u}=e^{-i\theta}{\bf u}\ \ \ \mbox{and}\ \ \ a_\alpha\mapsto \widetilde {a}_\alpha=a_\alpha+\partial_\alpha\theta,
\ 0\le\alpha\le n, \label{gauge-tran}
\end{equation}
for any $\theta\in C^\infty(\mathbb R^n\times [0, T])$.
A canonical choice is the Coulomb gauge which ensures that $\widetilde a=(\widetilde{a}_1,\cdots, \widetilde{a}_n)$ is divergence free:
\begin{equation}
{\rm{div}}\ \widetilde{a}=\sum_{k=1}^n \partial_k \widetilde{a}_k=0 \quad {\rm{in}}\quad \mathbb R^n. \label{gauge-tran1}
\end{equation}
This amounts to solve, for $t\in [0,T]$,
\begin{equation}
-\triangle\theta(t)={\rm{div}} (a(t)) \quad {\rm{in}}\quad \mathbb R^n, \label{gauge-tran2}
\end{equation}
whose solution is given by
$$\theta(x,t)=(-\Delta)^{-1}({\rm{div}} a)(x,t)=c_n \int_{\mathbb R^n}\frac{x-y}{|x-y|^n}\cdot a(y,t)\,dy.$$
Thus
$$\nabla_l\theta(t)=\nabla_l(-\triangle)^{-1}{\rm{div}} (a(t))=\sum_{j=1}^n{\bf R}_l{\bf R}_ja_j(t),\ 1\le l\le n,$$
where ${\bf R}_l=\nabla_l (-\Delta)^\frac12$ is the $l^{\rm{th}}$-Riesz transform (see \cite{Stein}).
Since $a, \partial_t a\in C^0\big([0,T], H^\infty_*(\mathbb R^n)\big)$, we conclude, by the standard elliptic theory,
that $\nabla\theta, \partial_t\theta\in C^0\big([0,T], H^\infty_*(\mathbb R^n)\big)$.

\subsection{Estimates of $a$ in Morrey spaces}

\begin{lemma}\label{le:4.1} For $p>2$ and $\frac{p}{2}<q\leq n$,
if $u(t)\in M^{p,q}(\mathbb R^n)$ and $\nabla u(t)\in M^{2,2}(\mathbb R^n)$ for $0<t\le T$,
then we have, under the Coulomb gauge,  that
\begin{equation}
\label{4.2}\Big\|a(t)\Big\|_{M^{\widetilde{p},q}(\mathbb R^n)}\lesssim\Big\|u(t)\Big\|_{M^{p,q}(\mathbb R^n)}^2, \ 0<t\le T,
\end{equation}
where $\displaystyle\widetilde{p}=\frac{pq}{2q-p}$. If, in addition, $q<\frac{2p}{4-p}$, then
 there is a decomposition $a_0(t)=a_0^{(1)}(t)+a_0^{(2)}(t)$ such that
\begin{equation}
\label{4.3}
\Big\|a_0^{(1)}(t)\Big\|_{M^{\frac{2p(p+q)}{(p+2)q},\frac{2(p+q)}{p+2}}(\mathbb R^n)}\lesssim\Big\|u(t)\Big\|_{M^{p,q}(\mathbb R^n)}\Big\|\nabla u(t)\Big\|_{M^{2,2}(\mathbb R^n)},
\ 0<t\le T,
\end{equation}
and
\begin{equation}
\label{4.4}
\Big\|a_0^{(2)}(t)\Big\|_{M^{\frac{\widetilde{p}}2,q}(\mathbb R^n)}\lesssim\Big\|u(t)\Big\|_{M^{p,q}(\mathbb R^n)}^4,
\ 0<t\le T.
\end{equation}
\end{lemma}
\begin{proof}
Since $a$ is a Coulomb gauge and
$$\partial_\alpha {a}_\beta(t)-\partial_\beta {a}_\alpha(t)={\rm{Im}}\big({u}_\alpha \overline{u}_\beta\big)(t)
\ \ {\rm{in}}\ \ \mathbb R^n, \ 0\le\alpha,\beta\le n,$$
we have, after taking $\partial_\alpha$ of the above equation and summing over $1\le\alpha\le n$, that
\begin{equation}
\label{4.5}-\triangle {a}_\beta(t)={\rm{div(Im}}\big({u}_\beta \overline{u})\big)(t) \ \ {\rm{in}}\ \ \mathbb R^n,
\ \ 0\le\beta\le n.
\end{equation}
It is readily seen that ${a}$ can be represented by
\begin{eqnarray*}
 {a}_\beta(x,t)=\int_{\mathbb R^n}G(x-y){\rm{div}}({\rm{Im}}({u}_\beta\overline{{u}}))(y,t)\,dy
=-\int_{\mathbb R^n} \nabla G(x-y)\cdot {\rm{Im}}({u}_\beta\overline{{u}})(y,t)\,dy, \ x\in\mathbb R^n,
\end{eqnarray*}
for $0\le\beta\le n$, where $G$ is the fundamental solution of ($-\Delta$) in $\mathbb R^n$.
Therefore we have
\begin{eqnarray*}
\big|{a}(x,t)\big|
\lesssim\int_{\mathbb R^n}|x-y|^{1-n}|{u}(t)|^2(y)\,dy:={\bf I}_1(|{u}(t)|^2)(x), \ x\in\mathbb R^n,
\end{eqnarray*}
where
$$\displaystyle {\bf I}_1(f)(x)=\int_{\mathbb R^n} |x-y|^{1-n} f(y)\,dy, \ f\in L^1_{\rm{loc}}(\mathbb R^n), $$
is  the Riesz potential of $f$ of order $1$.

Since ${\bf I}_1:M^{\frac{p}{2},q}(\mathbb R^n)\rightarrow M^{\widetilde{p},q}(\mathbb R^n)$ is
a bounded linear operator (see \cite{HW1}), we have that ${a}(t)\in M^{\widetilde{p}, q}(\mathbb R^n)$
and
$$
\Big\| {a}(t) \Big\|_{M^{\widetilde{p}, q}(\mathbb R^n)}\lesssim\Big\||{u}(t)|^2\Big\|_{M^{\frac{p}2, q}(\mathbb R^n)}
\lesssim \Big\|{u}(t)\Big\|_{M^{p, q}(\mathbb R^n)}^2.
$$
This implies (\ref{4.2}).

For $a_0$,  we have, by (\ref{3.1}),
\begin{eqnarray}
-\Delta a_0(t)&=&{\rm{div}}\big({\rm{Im}}(\overline{u}u_0)(t)\big)={\rm{div}}\big[{\rm{Im}}((\lambda-i)\overline{u}D_l u_l)(t)\big]\nonumber\\
&=&{\rm{div}}\big[{\rm{Im}}((\lambda-i)\overline{u}{\rm{div}}u)(t)+ {\rm{Re}}((\lambda-i)(a\cdot u)\overline{u})(t)\big].
\end{eqnarray}
We define $a_0^{(1)}$ and $a_0^{(2)}$ by
\begin{equation}
\label{4.6}
\begin{cases}
-\Delta a_0^{(1)}(t)={\rm{div}}\big[\lambda{\rm{Im}}(\overline{u}{\rm{div}}u)(t)-{\rm{Re}}(\overline{u}{\rm{div}}u)(t)\big],\\
-\Delta a_0^{(2)}(t)={\rm{div}}\big[\lambda{\rm{Re}}((a\cdot u)\overline{u})(t)+{\rm{Im}}((a\cdot u)\overline{u})(t)\big].
\end{cases}
\end{equation}
It is clear that $a_0(t)=a_0^{(1)}(t)+a_0^{(2)}(t)$. Direct calculations give
$$\big|a_0^{(1)}(x,t)\big|\lesssim {\bf I}_1\big(|u||\nabla u|\big)(x,t), \ x\in\mathbb R^n,$$
and
$$
\big|a_0^{(2)}(x,t)\big|\lesssim {\bf I}_1\big(|a||u|^2\big)(x,t), \ x\in\mathbb R^n.
$$
As above, we can show that $a_0^{(1)}(t)\in M^{\frac{2p(p+q)}{(p+2)q}, \frac{2(p+q)}{p+2}}(\mathbb R^n)$, and
$$\Big\|a_0^{(1)}(t)\Big\|_{M^{\frac{2p(p+q)}{(p+2)q}, \frac{2(p+q)}{p+2}}(\mathbb R^n)}
\lesssim \Big\|u(t)\Big\|_{M^{p,q}(\mathbb R^n)}\Big\|\nabla u(t)\Big\|_{M^{2,2}(\mathbb R^n)}.
$$
This yields (\ref{4.3}).
Similarly, we can show that $a_0^{(2)}(t)\in M^{\frac{\widetilde{p}}2, q}(\mathbb R^n)$ and
$$\Big\|a_0^{(2)}(t)\|_{M^{\frac{\widetilde{p}}{2}, q}(\mathbb R^n)}
\lesssim \Big\|a(t)\Big\|_{M^{\widetilde{p},q}(\mathbb R^n)}\Big\|u(t)\Big\|_{M^{p,q}(\mathbb R^n)}^2
\lesssim \Big\|u(t)\Big\|_{M^{p,q}(\mathbb R^n)}^4.$$
This yields (\ref{4.4}). This completes the proof.
\end{proof}

\setcounter{section}{3}
\setcounter{equation}{0}
\section{Estimates in Morrey spaces for the Covariant Landau-Lifshitz-Gilbert equation}

In this section, we consider a solution $u\in C^0([0,T], H^\infty(\mathbb R^n, \mathbb C^n))$ of the covariant complex Ginzburg-Landau equation (\ref{3.2}) under
the Coulomb gauge, and derive the necessary estimates in suitable Morrey spaces.

\subsection{Nonlinear estimates in Morrey spaces}
First recall that under the Coulomb gauge, the equation (\ref{3.2}) can be written as
\begin{equation}\label{CGL}
\partial_t u=(\lambda-i)\Delta u+F({\bf a}, u) \ \ {\rm{in}}\ \ \mathbb R^n\times (0,T],
\end{equation}
where the nonlinearity $F=(F_1,\cdots, F_n)$ is given by
\begin{equation}\label{F}
F_l({\bf a}, u)=(\lambda-i)\Big[i\sum_{k=1}^n\Big({\rm{Im}}(u_l\overline{u}_k)u_k\Big)+2i(a\cdot\nabla)u_l-|a|^2 u_l\Big]-i\big(a_0^{(1)}+a_0^{(2)}\big) u_l, \  \ 1\le l\le n.
\end{equation}
$F$ can be written as
$$F=f^{(1)}+f^{(2)}+f^{(3)},$$
where
\begin{equation}\label{F1}
\begin{cases}
f^{(1)}(u)=(\lambda-i)i\displaystyle\sum_{k=1}^n{\rm{Im}}(u\overline{u}_k)u_k,\\
f^{(2)}(u)=(\lambda-i)2i(a\cdot\nabla)u-i a_0^{(1)} u,\\
f^{(3)}(u)=-(\lambda-i)|a|^2 u-i a_0^{(2)} u.
\end{cases}
\end{equation}

For $p>2$ to be determined later, let $X_T^p$ be the function space defined by
\begin{eqnarray}
{\bf X}_T^p&:=&\Big\{u:\mathbb R^n\times[0,T)\rightarrow \mathbb C^{n}\ \Big|\ \
\big\|u \big\|_{{\bf X}_T^p}\equiv\sup\limits_{0<t<T}\Big(t^{\frac12-\frac{1}{p}}\big\|u(t)\big\|_{M^{p,2}(\mathbb R^n)}+t^\frac12\big\|\nabla u(t)\big\|_{M^{2,2}(\mathbb R^n)}\Big)
\nonumber\\
&&\qquad\qquad\qquad\qquad\qquad\qquad\qquad\qquad+\sup\limits_{0\le t<T}\Big\|u(t)\Big\|_{M^{2,2}(\mathbb R^n)}<+\infty\Big\}. \label{x-space}
\end{eqnarray}
For $u\in X_T^p$ and $0<\tau\le T$, set
$$\mathcal R_1(\tau)=\sup\limits_{0\leq t\leq \tau} t^{\frac12-\frac{1}{p}}\big\|u(t)\big\|_{M^{p,2}(\mathbb R^n)},$$
$$\mathcal R_2(\tau)=\sup\limits_{0\leq t\leq \tau} t^\frac12\big\|\nabla u(t)\big\|_{M^{2,2}(\mathbb R^n)},$$
and
$$\mathcal R_3(\tau)=\sup_{0\le t\le \tau}\big\|u(t)\big\|_{M^{2,2}(\mathbb R^n)}.$$
Then it holds
$$\big\|u\big\|_{X_T^p}=\mathcal R_1(T)+\mathcal R_2(T)+\mathcal R_3(T).
$$
For $\delta_1, \delta_2>0$, set
$$\mathbb B[\delta_1, \delta_2]:=\int_0^1 (1-t)^{-\delta_1}t^{-\delta_2}\,dt.$$
Observe that $\mathbb B[\delta_1, \delta_2]<+\infty$ when $\delta_1, \delta_2<1$.

\medskip
\noindent{4.1.1}. {\bf Morrey space estimates related to $f^{(1)}(u)$}.
Now we have
\begin{lemma}\label{le:5.4} For $p\in (3,6)$ and $u\in X_T^p,$ we have that for $0\le t\le T$,
\begin{equation}
\label{5.4} t^{\frac12-\frac{1}{p}}\Big\|\int_0^tS(t-s)f^{(1)}(u(s))\,ds\Big\|_{M^{p,2}(\mathbb R^n)}\lesssim \mathcal R_1^3(t).\end{equation}
\end{lemma}
\begin{proof} Since $|f^{(1)}(u)|\lesssim |u|^3$, by Minkowski's inequality and Lemma \ref{2.3} we have
\begin{eqnarray*}
\Big\|\int_0^tS(t-s)f^{(1)}(u(s))\,ds\Big\|_{M^{p,2}(\mathbb R^n)}&\lesssim&\int_0^t\Big\|S(t-s) f^{(1)}(u(s))\Big\|_{M^{p,2}(\mathbb R^n)}\,ds\\
&\lesssim&\big(\int_0^t(t-s)^{-\frac{2}{p}}s^{-\frac32(1-\frac{2}{p})}\,ds\big)\mathcal R_1^3(t)\\
&\lesssim&  \mathbb B\big[\frac{2}{p}, \frac32(1-\frac2{p})\big] t^{1-\frac{2}{p}-\frac32(1-\frac{2}{p})}\mathcal R_1^3(t)\\
&\lesssim &t^{\frac{1}{p}-\frac12}\mathcal R_1^3(t),
\end{eqnarray*}
since $\mathbb B\big[\frac{2}{p}, \frac32(1-\frac2{p})\big]<+\infty$ when $p\in (3,6)$.
Hence Lemma \ref{le:5.4} is proved.
\end{proof}

\begin{lemma}\label{lem5.7} For $p\in(3,6)$ and $u\in X^p_T$,  we have that for $0\le t\le T$,
\begin{equation}
\label{5.7} t^\frac12\Big\|\nabla\int_0^tS(t-s)f^{(1)}(u(s))\,ds\Big\|_{M^{2,2}(\mathbb R^n)}\lesssim \mathcal R_1^3(t).
\end{equation}
\end{lemma}
\begin{proof} By Minkowski's inequality and Lemma 2.6, we have
\begin{eqnarray*}
&&\Big\|\nabla\int_0^tS(t-s) f^{(1)}(u(s))\,ds\Big\|_{M^{2,2}(\mathbb R^n)}
\lesssim\int_0^t\Big\|\nabla(S(t-s)f^{(1)}(u(s)))\Big\|_{M^{2,2}(\mathbb R^n)}\,ds\\
&&\lesssim\int_0^t(t-s)^{-\frac12-(\frac{3}{p}-\frac{1}{2})}\Big\||u(s)|^3\Big\|_{M^{\frac{p}{3},2}(\mathbb R^n)}\,ds\\
&&\lesssim\int_0^t(t-s)^{-\frac{3}{p}}\Big\|u(s)\Big\|_{M^{p,2}(\mathbb R^n)}^3\,ds\\
&&\lesssim \mathcal R_1^3(t)\int_0^t(t-s)^{-\frac{3}{p}}s^{-\frac32(1-\frac{1}{p})}\,ds\\
&&\lesssim \mathbb B\big[\frac3p, \frac32(1-\frac2{p})\big] t^{-\frac{1}{2}} \mathcal R_1^3(t)\\
&&\lesssim t^{-\frac12}\mathcal R_1^3(t),
\end{eqnarray*}
since $\mathbb B\big[\frac3p, \frac32(1-\frac2{p})\big]<+\infty$ when $3<p<6$.
This implies (\ref{5.7}).  Hence Lemma \ref{lem5.7} is proved.
\end{proof}

\begin{lemma}\label{le:5.10} For $p\in (3,6)$ and $u\in X_T^p$, we have that for $0\le t\le T$,
\begin{equation}
\label{5.10}\Big\|\int_0^tS(t-s) f^{(1)}(u(s))\,ds\Big\|_{M^{2,2}(\mathbb R^n)}\lesssim \mathcal R_1^3(t).
\end{equation}
\end{lemma}
\begin{proof} By Minkowski's inequality and Lemma 2.3. we have
\begin{eqnarray*}
&&\Big\|\int_0^tS(t-s)f^{(1)}(u(s))\,ds\Big\|_{M^{2,2}(\mathbb R^n)}\lesssim\int_0^t\big\|S(t-s) f^{(1)}(u(s))\big\|_{M^{2,2}(\mathbb R^n)}\,ds\\
&&\lesssim\int_0^t(t-s)^{-(\frac{3}{p}-\frac{1}{2})}\Big\||u(s)|^3\Big\|_{M^{\frac{p}{3},2}(\mathbb R^n)}\,ds\\
&&\lesssim\int_0^t(t-s)^{-(\frac{3}{p}-\frac{1}{2})}\Big\|u(s)\Big\|_{M^{p,2}(\mathbb R^n)}^3\,ds\\
&&\lesssim \mathbb B\Big[\frac3p-\frac12, \frac32(1-\frac2{p})\Big] \mathcal R_1^3(t)\\
&&\lesssim \mathcal R_1^3(t),
\end{eqnarray*}
since $\mathbb B\Big[\frac3p-\frac12, \frac32(1-\frac2{p})\Big]<+\infty$ when $3<p<6$.
Hence Lemma \ref{le:5.10} is proved.
\end{proof}

\bigskip
\noindent{4.1.2}. {\bf Morrey space estimates related to $f^{(2)}(u)$}. Observe that
$$\big|f^{(2)}(u)\big|\lesssim |a||\nabla u|+|a_0^{(1)}||u|.$$
Now we have
\begin{lemma}\label{le:5.6} For $p\in (2,4)$ and $u\in X_T^p$, we have that for $0\le t\le T$,
\begin{equation}
\label{5.6}\Big\|\int_0^tS(t-s) f^{(2)}(u(s))\,ds\Big\|_{M^{p,2}(\mathbb R^n)}\lesssim t^{-(\frac12-\frac{1}{p})}\mathcal R_1^2(t)\mathcal R_2(t).
\end{equation}
\end{lemma}
\begin{proof} It follows from Lemma \ref{le:4.1} that  $a(t)\in M^{\frac{2p}{4-p},2}(\mathbb R^n)$ and $a_0^{(1)}(t)\in M^{p,2}(\mathbb R^n)$. Moreover,
$$\Big\|a(t)\Big\|_{M^{\frac{2p}{4-p},2}(\mathbb R^n)}\lesssim \Big\|u(t)\Big\|_{M^{p,2}(\mathbb R^n)}^2,$$
and
$$\Big\|a_0^{(1)}(t)\Big\|_{M^{p,2}(\mathbb R^n)}\lesssim \Big\|u(t)\Big\|_{M^{p,2}(\mathbb R^n)}\Big\|\nabla u(t)\Big\|_{M^{2,2}(\mathbb R^n)}.$$
By H\"older's inequality, we have that $f^{(2)}(u)(t)\in M^{\frac{p}2,2}(\mathbb R^n)$ and
\begin{eqnarray}\label{f2-estimate}
\Big\|f^{(2)}(u)(t)\Big\|_{M^{\frac{p}2,2}(\mathbb R^n)}
&\lesssim& \Big\|a(t)\Big\|_{M^{\frac{2p}{4-p},2}(\mathbb R^n)}\Big\|\nabla u(t)\Big\|_{M^{2,2}(\mathbb R^n)}+\Big\|a_0^{(1)}(t)\Big\|_{M^{p,2}(\mathbb R^n)}
\Big\|u(t)\Big\|_{M^{p,2}(\mathbb R^n)}\nonumber\\
&\lesssim& \Big\|u(t)\Big\|_{M^{p,2}(\mathbb R^n)}^2\Big\|\nabla u(t)\Big\|_{M^{2,2}(\mathbb R^n)}.
\end{eqnarray}
Applying Minkowski's inequality and Lemma \ref{le:2.3}, we then have
\begin{eqnarray*}
&&\Big\|\int_0^tS(t-s) f^{(2)}(u(s))\,ds\Big\|_{M^{p,2}(\mathbb R^n)}\leq\int_0^t\Big\|S(t-s) f^{(2)}(u(s))\Big\|_{M^{p,2}(\mathbb R^n)}\,ds\\
&&\lesssim \int_0^t(t-s)^{-\frac{1}{p}}\Big\|f^{(2)}(u(s))\Big\|_{M^{\frac{p}2, 2}(\mathbb R^n)}\,ds\\
&&\lesssim \int_0^t(t-s)^{-\frac{1}{p}}\Big\|u(s)\Big\|_{M^{p,2}(\mathbb R^n)}^2\Big\|\nabla u(s)\Big\|_{M^{2,2}(\mathbb R^n)}\,ds\\
&&\lesssim \mathcal R_1^2(t)\mathcal R_2(t)\int_0^t(t-s)^{-\frac1{p}}s^{-(\frac32-\frac{2}{p})}\,ds\\
&&\lesssim  \mathbb B\Big[\frac1p, \frac32-\frac2{p}\Big] t^{-(\frac12-\frac{1}{p})}\mathcal R_1^2(t)\mathcal R_2(t) \\
&&\lesssim t^{-(\frac12-\frac{1}{p})}\mathcal R_1^2(t),
\end{eqnarray*}
since $\mathbb B\Big[\frac1p, \frac32-\frac2{p}\Big]<+\infty$ when $2<p<4$.
This implies (\ref{5.6}). Hence Lemma \ref{le:5.6} is proved.
\end{proof}

\begin{lemma}\label{le:5.9} For $p\in(2,4)$ and $u\in X_T^p$, we have that for any $0\le t\le T$,
\begin{equation}
\label{5.9}t^{\frac12}\Big\|\nabla\int_0^tS(t-s) f^{(2)}(u(s))\,ds\Big\|_{M^{2,2}(\mathbb R^n)}\lesssim \mathcal R_1^2(t)\mathcal R_2(t).
\end{equation}
\end{lemma}
\begin{proof} Applying Minkowski's inequality, (\ref{f2-estimate}), and Lemma 2.6, we have
\begin{eqnarray*}
&&\Big\|\nabla\int_0^tS(t-s) f^{(2)}(u(s))\,ds\Big\|_{M^{2,2}(\mathbb R^n)}
\lesssim\int_0^t\Big\|\nabla(S(t-s)f^{(2)}(u(s)))\Big\|_{M^{2,2}(\mathbb R^n)}\,ds\\
&&\lesssim\int_0^t(t-s)^{-\frac12-(\frac{2}{p}-\frac{1}{2})}\Big\|f^{(2)}(u(s))\Big\|_{M^{p,2}(\mathbb R^n)}\,ds\\
&&\lesssim\int_0^t(t-s)^{-\frac{2}{p}}\big\|u(s)\big\|_{M^{p,2}(\mathbb R^n)}^2\big\|\nabla u(s)\big\|_{M^{2,2}(\mathbb R^n)}\,ds\\
&&\lesssim \mathcal R_1^2(t)\mathcal R_2(t)\int_0^t(t-s)^{-\frac{2}{p}}s^{-2(\frac12-\frac{1}{p})-\frac12}\,ds\\
&&\lesssim \mathbb B\Big[\frac{2}{p}, \frac32-\frac{2}{p}\Big] t^{-\frac12}\mathcal R_1^2(t)\mathcal R_2(t)\\
&&\lesssim t^{-\frac12}\mathcal R_1^2(t)\mathcal R_2(t),
\end{eqnarray*}
since $\mathbb B\Big[\frac{2}{p}, \frac32-\frac{2}{p}\Big]<+\infty$ when $2<p<4.$
This proves Lemma \ref{le:5.9}.
\end{proof}
\begin{lemma}\label{le:5.12} For $p\in(2,4)$ and $u\in X_T^p,$ we have that for $0\le t\le T$,
\begin{equation}\label{5.12}
\Big\|\int_0^tS(t-s) f^{(2)}(u(s))\,ds\Big\|_{M^{2,2}(\mathbb R^n)}\lesssim \mathcal R_1^2(t)\mathcal R_2(t).
\end{equation}
\end{lemma}

\begin{proof} Applying Minkowski's inequality, (\ref{f2-estimate}), and Lemma 2.6, we have
\begin{eqnarray*}
&&\Big\|\int_0^tS(t-s) f^{(2)}(u(s))\,ds\Big\|_{M^{2,2}(\mathbb R^n)}
\lesssim\int_0^t\Big\|S(t-s)f^{(2)}(u(s))\Big\|_{M^{2,2}(\mathbb R^n)}\,ds\\
&&\lesssim\int_0^t (t-s)^{-(\frac{2}{p}-\frac{1}{2})}\Big\|f^{(2)}(u(s))\Big\|_{M^{\frac{p}2,2}(\mathbb R^n)}\,ds\\
&&\lesssim\int_0^t(t-s)^{-(\frac{2}{p}-\frac{1}{2})}\Big\|u(s)\Big\|_{M^{p,2}(\mathbb R^n)}^2\Big\|\nabla u(s)\Big\|_{M^{2,2}(\mathbb R^n)}\,ds\\
&&\lesssim\int_0^t(t-s)^{-(\frac{2}{p}-\frac{1}{2})}\big\|u(s)\big\|_{M^{p,2}(\mathbb R^n)}^2\big\|\nabla u(s)\big\|_{M^{2,2}(\mathbb R^n)}\,ds\\
&&\lesssim \mathcal R_1^2(t)\mathcal R_2(t)\int_0^t(t-s)^{-(\frac{2}{p}-\frac{1}{2})}s^{-(\frac32-\frac{2}{p})}\,ds
=\mathbb B\Big[\frac{2}{p}-\frac12, \frac32-\frac{2}{p}\Big]\mathcal R_1^2(t)\mathcal R_2(t)\\
&&\lesssim \mathcal R_1^2(t)\mathcal R_2(t),
\end{eqnarray*}
since $\mathbb B\Big[\frac{2}{p}-\frac12, \frac32-\frac{2}{p}\Big]<+\infty$ when  $2<p<4$.
This proves Lemma \ref{le:5.12}.
\end{proof}

\medskip
\noindent{4.1.3}.  {\bf Morrey space estimates related to $f^{(3)}(u)$}. Observe that
$$\big|f^{(3)}(u)\big|\lesssim |a|^2|u|+|a_0^{(2)}|u|.$$
Now we have
\begin{lemma}\label{le:5.5} For $p\in(2,\frac{10}{3})$ and $u\in X_T^p,$ we have that for $0\le t\le T$,
\begin{equation}\label{5.5}
t^{\frac12-\frac{1}{p}}\Big\|\int_0^tS(t-s) f^{(3)}(u(s))\,ds\Big\|_{M^{p,2}(\mathbb R^n)}\lesssim \mathcal R_1^5(t).
\end{equation}
\end{lemma}
\begin{proof} By Lemma 3.2, we know that $a(t)\in M^{\frac{2p}{4-p},2}(\mathbb R^n)$ and $a_0^{(2)}(t)\in M^{\frac{p}{4-p}, 2}(\mathbb R^n)$
and
$$\big\|a(t)\big\|_{M^{\frac{2p}{4-p},2}(\mathbb R^n)}\lesssim\big\|u(t)\big\|_{M^{p,2}(\mathbb R^n)}^2,$$
$$\big\|a_0^{(2)}(t)\big\|_{M^{\frac{p}{4-p},2}(\mathbb R^n)}\lesssim\big\|u(t)\big\|_{M^{p,2}(\mathbb R^n)}^4.$$
Thus by H\"older's inequality we obtain that $f^{(3)}(u)\in M^{\frac{p}{5-p},2}(\mathbb R^n)$ and
\begin{eqnarray}\label{f3-estimate}
\Big\|f^{(3)}(u(t))\Big\|_{ M^{\frac{p}{5-p},2}(\mathbb R^n)}
&\lesssim& \Big(\Big\||a(t)|^2\Big\|_{M^{\frac{p}{4-p},2}(\mathbb R^n)}+\Big\|a_0^{(2)}(t)\Big\|_{M^{\frac{p}{4-p},2}(\mathbb R^n)}\Big)
\Big\|u(t)\Big\|_{M^{p,2}(\mathbb R^n)}\nonumber\\
&\lesssim& \Big\|u(t)\Big\|_{M^{p,2}(\mathbb R^n)}^5.
\end{eqnarray}
Applying Minkowski's inequality, (\ref{f3-estimate}), and Lemma 2.3, we have
\begin{eqnarray*}
&&\Big\|\int_0^tS(t-s)f^{(3)}(u(s))\,ds\Big\|_{M^{p,2}(\mathbb R^n)}\lesssim
\int_0^t\Big\|S(t-s)f^{(3)}(u(s))\Big\|_{M^{p,2}(\mathbb R^n)}\,ds\\
&&\lesssim\int_0^t(t-s)^{-\frac{4-p}{p}}\Big\|f^{(3)}(u(s))\Big\|_{M^{\frac{p}{5-p},2}(\mathbb R^n)}\,ds
\lesssim\int_0^t(t-s)^{-\frac{4-p}{p}}\Big\|u\Big\|_{M^{p,2}(\mathbb R^n)}^5\,ds\\
&&\lesssim \mathcal R_1^5(t)\int_0^t(t-s)^{-\frac{4-p}{p}}s^{-5(\frac12-\frac{1}{p})}\,ds\\
&&\lesssim \mathbb B\big[\frac{4-p}{p}, \frac52-\frac{5}{p}\big] t^{-(\frac12-\frac{1}{p})}\mathcal R_1^5(t)\\
&&\lesssim t^{-(\frac12-\frac{1}{p})}\mathcal R_1^5(t),
\end{eqnarray*}
since $\mathbb B\big[\frac{4-p}{p}, \frac52-\frac{5}{p}\big]<+\infty$
when  $2<p<\frac{10}{3}$.
This completes the proof.
\end{proof}

\begin{lemma}\label{le:5.8} For $p\in(\frac52,\frac{10}{3})$ and $u\in X_T^p$, we have that for $0\le t\le T$,
\begin{equation}
\label{5.8}t^{\frac12}\Big\|\nabla\int_0^tS(t-s)f^{(3)}(u(s))\,ds\Big\|_{M^{2,2}(\mathbb R^n)}\lesssim \mathcal R_1^5(t).
\end{equation}
\end{lemma}
\begin{proof} Applying Minkowski's inequality, (\ref{f3-estimate}), and Lemma 2.6, we have
\begin{eqnarray*}
&&\Big\|\nabla\int_0^tS(t-s)f^{(3)}(u(s)) \,ds\Big\|_{M^{2,2}(\mathbb R^n)}\lesssim\int_0^t\Big\|\nabla(S(t-s)f^{(3)}(u(s)))\Big\|_{M^{2,2}(\mathbb R^n)}\,ds\\
&&\lesssim\int_0^t(t-s)^{-\frac12-(\frac{5-p}{p}-\frac{1}{2})}\Big\|f^{(3)}(u(s))\Big\|_{M^{\frac{p}{5-p},2}(\mathbb R^n)}\,ds\\
&&\lesssim\int_0^t(t-s)^{-\frac{5-p}{p}}s^{5(\frac12-\frac{1}{p})}s^{-5(\frac12-\frac{1}{p})}\Big\|u(s)\Big\|_{M^{p,2}(\mathbb R^n)}^5\,ds\\
&&\lesssim\mathbb B\big[\frac{5-p}{p}, \frac{5}{2}-\frac{5}{p}\big] t^{-\frac12} \mathcal R_1^5(t)\\
&&\lesssim t^{-\frac12}\mathcal R_1^5(t) ,
\end{eqnarray*}
since $\mathbb B\big[\frac{5-p}{p}, \frac{5}{2}-\frac{5}{p}\big]<+\infty$ when $\frac52<p<\frac{10}{3}$.
Hence Lemma \ref{le:5.8} is proved.
\end{proof}

\begin{lemma}\label{le:5.11} For $p\in(2,\frac{10}{3})$ and $u\in X_T^p,$ we have that for $0\le t\le T$,
\begin{equation}\label{5.11}
\Big\|\int_0^tS(t-s)f^{(3)}(u(s))\,ds\Big\|_{M^{2,2}(\mathbb R^n)}\lesssim \mathcal R_1^5(t).
\end{equation}
\end{lemma}

\begin{proof}  Applying Minkowski's inequality, (\ref{f3-estimate}), and Lemma 2.6, we have
\begin{eqnarray*}
&&\Big\|\int_0^tS(t-s)f^{(3)}(u(s))\,ds\Big\|_{M^{2,2}(\mathbb R^n)}\lesssim\int_0^t\big\|S(t-s)f^{(3)}(u(s))\big\|_{M^{2,2}(\mathbb R^n)}\,ds\\
&&\lesssim\int_0^t(t-s)^{-(\frac{5-p}{p}-\frac{1}{2})}\big\|f^{(3)}(u(s))\big\|_{M^{\frac{p}{5-p},2}(\mathbb R^n)}\,ds\\
&&\lesssim\int_0^t(t-s)^{-(\frac{5-p}{p}-\frac{1}{2})}s^{5(\frac12-\frac{1}{p})}s^{-5(\frac12-\frac{1}{p})}\big\|u(s)\big\|_{M^{p,2}(\mathbb R^n)}^5\,ds\\
&&\lesssim \mathbb B\big [\frac{5}{p}-\frac32, \frac52-\frac{5}{p}\big]\mathcal R_1^5(t)\\
&&\lesssim \mathcal R_1^5(t),
\end{eqnarray*}
since $\mathbb B\big [\frac{5}{p}-\frac32, \frac52-\frac{5}{p}\big]<+\infty$ when $2<p<\frac{10}{3}$.
This proves Lemma \ref{le:5.11}.
\end{proof}
\subsection{Duhamel's principle and Morrey space estimates of solutions to (\ref{3.2})} Assume that
$$u\in C^0\big([0,T], H^\infty(\mathbb R^n, \mathbb C^n)\big)$$
solves the covariant complex Ginzburg-Landau equation (\ref{3.2}). By Duhamel's formula, we have that for $0\le t\le T$,
\begin{equation}\label{duhamel1}
u(t)=S(t)u_0+\sum_{l=1}^3 (S*f^{(l)}(u))(t),
\end{equation}
where
\begin{equation}\label{duhamel2}
S(t)u_0=S_t*u_0:=\widetilde{u_0}(t)
\end{equation}
is, as before, the convolution in space, and
\begin{equation}\label{duhamel3}
 (S*f^{(l)}(u))(t)=\int_0^t S(t-s)(f^{(l)}(u(s)))\,ds:=u^{(l)}(t), \ 1\le l\le 3.
\end{equation}

For $\widetilde{u_0}$, we can apply Lemma 2.2, Lemma 2.3 and Lemma 2.6 to get
\begin{lemma}\label{x-estimate} For any $2<p<+\infty$, there exists $C>0$ depending only on $n, p, \lambda$
such that for any $0\le t<+\infty$,
\begin{equation}\label{x-estimate1}
\Big\|\widetilde{u_0}\Big\|_{X_t^p}\le C \Big\|u_0\Big\|_{M^{2,2}(\mathbb R^n)}.
\end{equation}
\end{lemma}
\begin{proof} By Lemma 2.3, we have that for $0\le t\le T$,
\begin{eqnarray*}
\Big\|\widetilde{u_0}(t)\Big\|_{M^{p,2}(\mathbb R^n)}&=&\Big\|S(t) u_0\Big\|_{M^{p,2}(\mathbb R^n)}
\lesssim  t^{-(\frac12-\frac1{p})}\Big\|u_0\Big\|_{M^{2,2}(\mathbb R^n)}.
\end{eqnarray*}
By Lemma 2.6, we have that for $0\le t\le T$,
\begin{eqnarray*}
\Big\|\nabla\widetilde{u_0}(t)\Big\|_{M^{2,2}(\mathbb R^n)}&=&\Big\|\nabla(S(t) u_0)\Big\|_{M^{2,2}(\mathbb R^n)}
\lesssim  t^{-\frac12}\Big\|u_0\Big\|_{M^{2,2}(\mathbb R^n)}.
\end{eqnarray*}
By Lemma 2.2, we  have that for $0\le t\le T$,
\begin{eqnarray*}
\Big\|\widetilde{u_0}(t)\Big\|_{M^{2,2}(\mathbb R^n)}&=&\Big\|S(t) u_0\Big\|_{M^{2,2}(\mathbb R^n)}
\lesssim  \Big\|u_0\Big\|_{M^{2,2}(\mathbb R^n)}.
\end{eqnarray*}
Combining these estimates together yields
\begin{eqnarray*}
\Big\|\widetilde{u_0}\Big\|_{X_t^p}&=& \sup_{0\le\tau\le t}\Big[\tau^{\frac12-\frac1{p}}\Big\|\widetilde{u_0}(\tau)\Big\|_{M^{p,2}(\mathbb R^n)}
+\tau^\frac12\Big\|\nabla\widetilde{u_0}(\tau)\Big\|_{M^{2,2}(\mathbb R^n)}+\Big\|\widetilde{u_0}(\tau)\Big\|_{M^{2,2}(\mathbb R^n)}\Big]\\
&\lesssim &  \Big\|u_0\Big\|_{M^{2,2}(\mathbb R^n)}.
\end{eqnarray*}
This proves (\ref{x-estimate1}).
\end{proof}

We also introduce another norm of $u\in X_T^p$:
$$\Big\|u\Big\|_{Y_t^p}:
=\sup\limits_{0<\tau<t}\Big(\tau^{\frac12-\frac{1}{p}}\big\|u(\tau)\big\|_{M^{p,2}(\mathbb R^n)}+\tau^\frac12\big\|\nabla u(\tau)\big\|_{M^{2,2}(\mathbb R^n)}\Big),
\ 0\le t\le T.
$$
Note that
$$\Big\|u\Big\|_{Y_t^p}=\mathcal R_1(t)+\mathcal R_2(t), \ \Big\|u\Big\|_{X_t^p}=\Big\|u\Big\|_{Y_t^p}+\mathcal R_3(t) \big(\ge \Big\|u\Big\|_{Y_t^p}\big), \ 0\le t\le T.$$
Now we can combine all these Lemmas together to obtain the following key estimate.
\begin{thm}\label{x-space-estimate2} For any $p\in (3, \frac{10}3)$ and $0<T\le +\infty$, there exists
a constant $C>0$ depending only on $n, p, \lambda$ such that
if $u\in C^0\big([0,T], H^\infty(\mathbb R^n, \mathbb C^n)\big)$ solves the covariant Ginzburg-Landau
equation (\ref{3.2}). Then
\begin{equation}\label{x-space-estimate3}
\Big\|u\Big\|_{X_t^p}\leq \Big\|\widetilde{u_0}\Big\|_{X^p_t}
+C\Big(\big\|u\big\|_{Y_t^p}^3+\big\|u\big\|_{Y_t^p}^5\Big), \ \forall\ 0\le t\le T.
\end{equation}
Furthermore, there exists $\epsilon_0>0$ depending on $n, p,\lambda$ such that if
\begin{equation}\label{small_initial}
\Big\|u_0\Big\|_{M^{2,2}(\mathbb R^n)}\le\epsilon_0,
\end{equation}
then for any $0\le t\le T$,
\begin{equation}\label{x-space-estimate4}
\Big\|u\Big\|_{X_t^p}\le 2\Big\|\widetilde{u_0}\Big\|_{X_t^p}\le C\Big\|u_0\Big\|_{M^{2,2}(\mathbb R^n)}.
\end{equation}
\end{thm}
\begin{proof} By the Duhamel formula (\ref{duhamel1}), we have
$$
\Big\|u\Big\|_{X_t^p}\leq \Big\|\widetilde{u_0}\Big\|_{X^p_t}+\sum_{l=1}^3 \Big\|u^{(l)}\Big\|_{X_t^p}.
$$
 For $u^{(1)}$, since Lemma 4.1, Lemma 4.2, and Lemma 4.3 hold for $p\in (3, \frac{10}3)$, we have
$$\Big\|u^{(1)}\Big\|_{X_t^p}\le C\mathcal R_1^3(t)\le C\Big\|u\Big\|_{Y_t^p}^3.$$
 For $u^{(2)}$, since Lemma 4.4, Lemma 4.5, and Lemma 4.6 hold for $p\in (3, \frac{10}3)$, we have
$$\Big\|u^{(2)}\Big\|_{X_t^p}\le C\mathcal R_1^2(t)\mathcal R_2(t)\le C\Big\|u\Big\|_{Y_t^p}^3.$$
For $u^{(3)}$, since Lemma 4.7, Lemma 4.8. and Lemma 4.9 hold for $p\in (3, \frac{10}3)$, we have
$$\Big\|u^{(3)}\Big\|_{X_t^p}\le C\mathcal R_1^5(t)\le C\Big\|u\Big\|_{Y_t^p}^5.$$
Putting these estimates together yields (\ref{x-space-estimate3}).

To show (\ref{x-space-estimate4}), first observe that
\begin{equation}\label{initial_vanish}
\lim_{t\downarrow 0^+} \Big\|u\Big\|_{Y_t^p}=0.
\end{equation}
In fact, by Sobolev's embedding theorem we have the following estimates: For some large $\sigma\in \mathbb Z_+$,
\begin{eqnarray*}
\Big\|u(t)\Big\|_{M^{p,2}(\mathbb R^n) }&\lesssim&\sup_{x\in\mathbb R^n} \Big\|u(t)\Big\|_{L^\infty(B_1(x))}+\sup_{x\in\mathbb R^n, R>1} \Big(R^{2-n}\int_{B_R(x)}|u(t)|^p\Big)^\frac1p
\lesssim  \Big\|u(t)\Big\|_{H^\sigma(\mathbb R^n)},
\end{eqnarray*}
and
\begin{eqnarray*}
\Big\|\nabla u(t)\Big\|_{M^{2,2}(\mathbb R^n) }&\lesssim&\sup_{x\in\mathbb R^n} \Big\|\nabla u(t)\Big\|_{L^\infty(B_1)}+\sup_{x\in\mathbb R^n, R>1} \Big(R^{2-n}\int_{B_R(x)}|\nabla u(t)|^2\Big)^\frac12
\lesssim  \Big\|u(t)\Big\|_{H^\sigma(\mathbb R^n)}.
\end{eqnarray*}
This, combined with $u\in C^0([0, T], H^\infty(\mathbb R^n,\mathbb C^n))$,  implies (\ref{initial_vanish}).

It follows from (\ref{initial_vanish}) and (\ref{x-space-estimate3}) that $ \big\|u\big\|_{X_{t}^p}<2\big\|\widetilde{u_0}\big\|_{X_{t}^p}$ for sufficiently small $t>0$.
Now, assume that there exists $t_*\in (0, T)$ such that $\displaystyle \big\|u\big\|_{X_{t_*}^p}=2\big\|\widetilde{u_0}\big\|_{X_{t_*}^p}\not=0$. This,
combined with $\displaystyle \big\|u\big\|_{Y_{t_*}^p}\le\displaystyle \big\|u\big\|_{X_{t_*}^p}$ and (\ref{x-space-estimate3}), implies that
\begin{eqnarray*}2\big\|\widetilde{u_0}\big\|_{X_{t_*}^p}&=&\big\|u\big\|_{X_{t_*}^p}
\le \Big\|\widetilde{u_0}\Big\|_{X^p_{t_*}}
+C\Big(\big\|u\big\|_{Y_{t_*}^p}^3+\big\|u\big\|_{Y_{t_*}^p}^5\Big)\\
&=& \Big\|\widetilde{u_0}\Big\|_{X^p_{t_*}}+C\Big(8\big\|{\widetilde{u_0}}\big\|_{X_{t_*}^p}^2+32\big\|\widetilde{u_0}\big\|_{X_{t_*}^p}^4\Big)\Big\|\widetilde{u_0}\Big\|_{X^p_{t_*}}.
\end{eqnarray*}
Hence we obtain
$$1\le 8C\big\|\widetilde{u_0}\big\|_{X_{t_*}^p}^2+32C\big\|\widetilde{u_0}\big\|_{X_{t_*}^p}^4.$$
This and Lemma \ref{x-estimate} imply
$$1\le 8C^3\big\|u_0\big\|_{M^{2,2}(\mathbb R^n)}^2+32C^5\big\|u_0\big\|_{M^{2,2}(\mathbb R^n)}^4
\le 8C^3\epsilon_0^2+32C^5\epsilon_0^4,
$$
which is impossible, provided that $\epsilon_0>0$ is chosen sufficiently small. Hence (\ref{x-space-estimate4}) holds and the proof is complete.
\end{proof}

\section{Global well-posedness of LLG (\ref{1.1}) and proof of Theorem \ref{th:1.1}}

In this section, we will give a proof of Theorem \ref{th:1.1}. The rough idea is to (i) approximate the initial data ${\bf m}_0$ by
${\bf m}_0^k$ in $H^\infty_*(\mathbb R^n, \mathbb S^2)$ and consider the local smooth solutions ${\bf m}^k:\mathbb R^n\times [0, T^k]\to\mathbb S^2$
of (\ref{1.1}) with initial data ${\bf m}_0^k$; (ii) apply Theorem \ref{covariant_LLG} and Theorem \ref{x-space-estimate2} to obtain uniform bounds on $\displaystyle
\big\|\nabla {\bf m}^k\big\|_{X_{T^k}^p}$;  (iii) employ the $\epsilon$-regularity theory of LLG (\ref{1.1}) to obtain $T^k=+\infty$
and uniform upper bounds on $\displaystyle\big\|{\bf m}^k\big\|_{C^l(\mathbb R^n\times [\delta, +\infty))}$ for all $l\in\mathbb Z_+$ and $\delta>0$; and
(iv) show that the limit map ${\bf m}$ of ${\bf m}^k$ is the desired solution.

\subsection{Approximation of initial data} For an initial data ${\bf m}_0$ given by Theorem \ref{th:1.1}, we have
\begin{lemma}\label{approx} There exists $\epsilon_0>0$ such that if ${\bf m}_0:\mathbb R^n\to\mathbb S^2$ satisfies ${\bf m}_0-{\bf m}_\infty\in L^2(\mathbb R^n)$
for some ${\bf m}_\infty\in\mathbb S^2$, and
\begin{equation}\label{approx1}
\displaystyle\big\|\nabla{\bf m}_0\big\|_{M^{2,2}(\mathbb R^n)}\le\epsilon_0,
\end{equation}
then there exist a sequence of maps  $\{{\bf m}_0^k\}\subset H^\infty_*(\mathbb R^n,\mathbb S^2)$ such that
\begin{equation}\label{approx2}
\sup_{k\ge 1}\big\|\nabla {\bf m}_0^k\big\|_{M^{2,2}(\mathbb R^n)}\le C\big\|\nabla{\bf m}_0\big\|_{M^{2,2}(\mathbb R^n)}\le C\epsilon_0,
\end{equation}
and
\begin{equation}\label{approx3}
\lim_{k\rightarrow \infty}\Big(\big\|{\bf m}_0^k-{\bf m}_0\big\|_{L^2(\mathbb R^n)}+\big\|\nabla({\bf m}_0^k-{\bf m}_0)\big\|_{L^2(B_R)}\Big)=0, \ \forall R>0.
\end{equation}
If, in addition, $\nabla{\bf m}_0\in L^2(\mathbb R^n)$, then
\begin{equation}\label{h1-conv}
\lim_{k\rightarrow\infty}\big\|\nabla({\bf m}_0^k-{\bf m}_0)\big\|_{L^2(\mathbb R^n)}=0.
\end{equation}
\end{lemma}
\begin{proof} Let $\phi\in C^\infty(\mathbb R^n,\mathbb R_+)$ be a standard mollifier, with supp ($\phi)\subset B_1$ and $\displaystyle\int_{\mathbb R^n}\phi(y)\,dy=1$.
Set $\displaystyle\phi_k(x)=k^n\phi(kx)$ and define
$$\widetilde{\bf m}_0^k(x)=(\phi_k*{\bf m}_0)(x)=\int_{\mathbb R^n}\phi_k(x-y){\bf m}_0(y)\,dy=\int_{\mathbb R^n}\phi(y) {\bf m}_0(x-\frac1{k}y)\,dy, \ x\in\mathbb R^n.$$
Applying a modified Poincar\'e inequality, we have that for any $x\in \mathbb R^n$ and $k\ge 1$,
\begin{eqnarray*}
\int_{B_1(x)}\big|\widetilde{\bf m}_0^k(x)-{\bf m}_0(x-\frac1{k}y)\big|^2\,dy
&\lesssim& \int_{B_1(x)}|\nabla\widetilde{\bf m}_0^k|^2\\
&\lesssim& 2^{2-n}\int_{B_{2}(x)}|\nabla {\bf m}_0|^2\le C\big\|\nabla{\bf m}_0\big\|_{M^{2,2}(\mathbb R^n)}^2.
\end{eqnarray*}
Therefore we have that for any $x\in\mathbb R^n$ and $k\ge 1$, 
$$
{\rm{dist}}^2(\widetilde{\bf m}_0^k(x),\mathbb S^2)\le C\int_{B_1(x)}\big|\widetilde{\bf m}_0^k(x)-{\bf m}_0(x-\frac1{k}y)\big|^2\,dy
\le C\big\|\nabla{\bf m}_0\big\|_{M^{2,2}(\mathbb R^n)}^2\le C\epsilon_0^2.
$$
By choosing $\epsilon_0>0$ sufficiently small, we have that
$$\frac34\le |\widetilde{\bf m}_0^k(x)|\le 1, \ \forall x\in\mathbb R^n.$$
Since $\widetilde{\bf m}_0^k-{\bf m}_\infty=\phi_k*({\bf m}_0-{\bf m}_\infty)$, it is not hard to see that
$$ \lim_{k\rightarrow +\infty}\big\|\widetilde{\bf m}_0^k-{\bf m}_0\big\|_{L^2(\mathbb R^n)}
= \lim_{k\rightarrow \infty}\big\|\phi_k*({\bf m}_0-{\bf m}_\infty)-({\bf m}_0-{\bf m}_\infty)\big\|_{L^2(\mathbb R^n)}=0.$$
It is easy to check that for any $k\ge 1$,
$$R^{2-n}\int_{B_R(x)}\big|\nabla\widetilde{\bf m}_0^k\big|^2\le R^{2-n}\int_{B_{R+1}(x)}|\nabla {\bf m}_0|^2, \ \forall B_R(x)\subset\mathbb R^n,$$
so that
$$\Big\|\nabla \widetilde{\bf m}_0^k\Big\|_{M^{2,2}(\mathbb R^n)}\le C\Big\|\nabla {\bf m}_0\Big\|_{M^{2,2}(\mathbb R^n)}.
$$
Set $$\displaystyle \Pi(y)=\frac{y}{|y|}:\Big\{y\in\mathbb R^3\ | \ \frac34\le |y|\le 1\Big\} \to \mathbb S^2,$$ and define
$${\bf m}_0^k(x):=\Pi(\widetilde{\bf m}_0^k)(x)\big(=\frac{\widetilde{\bf m}_0^k(x)}{|\widetilde{\bf m}_0^k(x)|}\big), \ \forall x\in \mathbb R^n.$$
Since
$$\Big\|\nabla\Pi\Big\|_{L^\infty(\{y\in\mathbb R^3\ | \ \frac34\le |y|\le 1\})}\le 4,$$
it is easy to check that  ${\bf m}_0^k\in C^\infty(\mathbb R^n,\mathbb S^2)$ satisfies (\ref{approx2}).
Since
$$\displaystyle{\bf m}_0^k(x)-{\bf m}_0(x)=\Pi(\widetilde{\bf m}_0^k(x))-\Pi({\bf m}_0(x)), \ \forall x\in\mathbb R^n,$$
it is also easy to see that
$$\lim_{k\rightarrow\infty}\big\|{\bf m}_0^k-{\bf m}_0\big\|_{L^2(\mathbb R^n)}=0,$$
and
$$\lim_{k\rightarrow\infty}\big\|\nabla({\bf m}_0^k-{\bf m}_0)\big\|_{L^2(B_R)}=0$$
for any $0<R<+\infty$.
In particular, we have that ${\bf m}_0^k-{\bf m}_\infty\in L^2(\mathbb R^n)$. For any $l\ge 1$,
since $\nabla^l \widetilde {\bf m}_0^k=\nabla^l(\widetilde{\bf m}_0^k-{\bf m}_\infty)=\nabla^l(\phi_k*({\bf m}_0-{\bf m}_\infty))\in L^2(\mathbb R^n)$
and $\nabla^l{\bf m}_0^k=\nabla^l\big(\Pi(\widetilde{\bf m}_0^k)\big)$, we conclude that $\nabla^l {\bf m}_0^k\in L^2(\mathbb R^n)$.
Thus ${\bf m}_0^k\in H^\infty_*(\mathbb R^n,\mathbb S^2)$.

It is clear that when $\nabla{\bf m}_0\in L^2(\mathbb R^n)$, $\nabla({\bf m}_0^k-{\bf m}_0)\rightarrow 0$ in $L^2(\mathbb R^n)$.
The proof is complete.
\end{proof}

\subsection{Uniform estimates of approximate solutions of (\ref{1.1})} For ${\bf m}_0$ given by Theorem \ref{th:1.1},
let ${\bf m}_0^k$ be the approximated initial data given by Lemma \ref{approx}. Then it is known  (see \cite{Melcher} proposition 1)
that  there exist a maximal time interval $T^k>0$ and a solution
$${\bf m}^k\in C^0([0, T^k], H^\infty_*(\mathbb R^n,\mathbb S^2)), \ {\rm{with}}\ \partial_t {\bf m}^k\in C^0([0, T^k], H^\infty(\mathbb R^n)),$$
to the LLG (\ref{1.1}), with initial data ${\bf m}_0^k$. Now let $({\bf a}^k, {\bf u}^k)\in C^0([0,T^k], H^\infty(\mathbb R^n, \mathbb R^{n+1}\times \mathbb C^{n+1}))$
be the corresponding solution of the covariant LLG equation and the covariant Ginzburg-Landau equation given by Theorem 3.1, under the Coulomb gauge. Since
$$\big\|u_0^k\big\|_{M^{2,2}(\mathbb R^n)}\approx \big\|\nabla {\bf m}_0^k\big\|_{M^{2,2}(\mathbb R^n)}\le \epsilon_0,$$
we can apply Theorem \ref{x-space-estimate2} to conclude that for $p\in (3,\frac{10}3)$, $u^k\in X_{T^k}^p$ and satisfies
\begin{equation}\label{x-estimate6}
\big\|u^k\big\|_{X^p_{T^k}}\le C\big\|u_0^k\big\|_{M^{2,2}(\mathbb R^n)}\le C\epsilon_0.
\end{equation}
By virtue of the relation between $u^k$ and ${\bf m}^k$, this implies
\begin{equation}\label{x-estimate7}
\sup_{0\le t\le T_k}\big\|\nabla {\bf m}^k(t)\big\|_{M^{2,2}(\mathbb R^n)}\le \big\|u^k\big\|_{X^p_{T_k}}\le C\big\|u_0^k\big\|_{M^{2,2}(\mathbb R^n)}\le C\epsilon_0.
\end{equation}

\subsection{$\epsilon_0$-regularity of LLG (\ref{1.1})} To show that (\ref{x-estimate7}) yields $T^k=+\infty$ and ${\bf m}^k\in C^\infty(\mathbb R^n\times (0,+\infty))$ with uniformly bounded $C^l$-norms in $\mathbb R^n\times [\delta,+\infty)$
for any $\delta>0$, we need the following regularity theorem for LLG (\ref{1.1}). For $z_0=(x_0,t_0)\in\mathbb R^n\times (0,+\infty)$ and $0<r_0<\sqrt{t_0}$,
let $P_{r_0}(z_0)=B_r(x_0)\times [t_0-r_0^2, t_0]$ denote the parabolic ball with center $z_0$ and radius $r_0$. Set
$$H^1(P_{r_0}(z_0),\mathbb S^2)=\Big\{{\bf m}: P_{r_0}(z_0)\to \mathbb S^2\ \Big | \ \nabla {\bf m}\in L^2(P_{r_0}(z_0)), \ \partial_t{\bf m}\in L^2(P_{r_0}(z_0))\Big\}.$$
\begin{lemma}\label{e-regularity}  There exists $\epsilon_0>0$ depending only on $\lambda$ and $n$ with the following property:
If ${\bf m}\in  H^1(P_{r_0}(z_0),\mathbb S^2)$, with $\nabla^2{\bf m}\in L^2(P_{r_0}(z_0))$,  is a weak solution
of the LLG (\ref{1.1})\footnote{${\bf m}$ satisfies (\ref{1.1}) in the sense of distributions, if,  for any $\Phi\in C_0^\infty(P_{r_0}(z_0),\mathbb S^2)$,
the following holds:
$$\int_{P_{r_0}(z_0)} \big[\langle\partial_t {\bf m}, \Phi\rangle+\lambda\langle\nabla{\bf m}, \nabla\Phi\rangle-\langle{\bf m}\times\nabla{\bf m},\nabla\Phi\rangle
-\lambda\langle|\nabla{\bf m}|^2{\bf m},\Phi\rangle\big]\,dyds=0.$$} satisfying
\begin{equation}\label{e-condition}
\Big\|\nabla{\bf m}\Big\|_{M^{2,2}(P_{r_0}(z_0))}:=\sup_{P_r(z)\subset P_{r_0}(z_0)} \Big( r^{2-(n+2)}\int_{P_r(z)}|\nabla{\bf m}(y,s)|^2\,dyds\Big)^\frac12
\le\epsilon_0,
\end{equation}
then ${\bf m}\in C^\infty(P_{\frac{r_0}2}(z_0), \mathbb S^2)$,  and for all $l\ge 0$,
\begin{equation}\label{e-conclusion}
\big\|{\bf m}\big\|_{C^l(P_{\frac{r_0}2}(z_0))}\le C(\lambda, l, \epsilon_0, r_0).
\end{equation}
\end{lemma}
\begin{proof} By rescaling, we may assume $r_0=2$ and $z_0=(0,4)$. Since $\nabla^2 {\bf m}\in L^2(P_2(0,4))$, it is known that
(\ref{1.1}) is equivalent to
\begin{equation}\label{LLG}
\lambda \partial_t{\bf m}+{\bf m}\times \partial_t {\bf m}=(1+\lambda^2)(\Delta{\bf m}+|\nabla{\bf m}|^2{\bf m}).
\end{equation}
For any $\phi\in C_0^\infty(B_2)$, we can multiply (\ref{LLG}) by $\partial_t{\bf m}\phi^2$ and integrate the resulting equation
over $B_2\times [t_1, t_2]$, $0<t_1<t_2\le 4$,  to obtain
the following local energy inequality:
\begin{equation}\label{local_energy_ineq1}
\lambda\int_{t_1}^{t_2}\int_{B_4} |\partial_t{\bf m}|^2\phi^2+(1+\lambda^2)\int_{B_2}|\nabla{\bf m}(t_2)|^2\phi^2
\le (1+\lambda^2)\int_{B_2}|\nabla{\bf m}(t_1)|^2\phi^2+C(\lambda)\int_{t_1}^{t_2}\int_{B_2}|\nabla{\bf m}|^2|\nabla\phi|^2.
\end{equation}
By choosing suitable $\phi$ and applying Fubini's theorem, (\ref{local_energy_ineq1}) implies that
\begin{equation}\label{local_energy_ineq2}
\int_{P_{\frac{r}2}(z)}|\partial_t{\bf m}|^2\le C(\lambda) r^{-2}\int_{P_{r}(z)}|\nabla{\bf m}|^2, \ \forall P_r(z)\subset P_2(0,4).
\end{equation}
This, combined with (\ref{e-condition}), implies the smallness of renormalized total energy of ${\bf m}$:
\begin{equation}\label{e-condition1}
r^{-n}\int_{P_r(z)}(|\nabla{\bf m}|^2+r^2|\partial_t{\bf m}|^2)\le C\epsilon_0^2, \ \forall \ P_r(z)\subset P_{\frac32}(0,4).
\end{equation}
Now we can perform a blow-up argument, similar to that by
Moser \cite{Moser}, Melcher \cite{Melcher0}, and Ding-Wang \cite{Ding-Wang},
to conclude that there exists $\theta_0\in (0,\frac14)$ such that
\begin{equation}\label{e-conclusion1}
(\theta_0 r)^{-n}\int_{P_{\theta_0 r}(z)}(|\nabla{\bf m}|^2+(\theta_0r)^2|\partial_t{\bf m}|^2)\le
\frac12 r^{-n}\int_{P_{r}(z)}(|\nabla{\bf m}|^2+r^2|\partial_t{\bf m}|^2),
\ \forall \ P_r(z)\subset P_{\frac32}(0,4).
\end{equation}
It is standard that by iterations, (\ref{e-conclusion1}) implies that there exists $\alpha_0\in (0,1)$ depending only on
$n,\lambda$ such that
\begin{equation}\label{morrey1}
r^{-n}\int_{P_{r}(z)}(|\nabla{\bf m}|^2+r^2|\partial_t{\bf m}|^2)
\le C(\lambda) r^{2\alpha_0}\int_{P_{2}(0,4)}|\nabla{\bf m}|^2 \le C(\lambda) \epsilon_0^2 r^{2\alpha_0},
\ \forall\ P_r(z)\subset P_{\frac32}(0,4).
\end{equation}
This, combined with Morrey's decay lemma, yields that ${\bf m}\in C^{\alpha_0}\big(P_{\frac32}(0,4),\mathbb S^2\big)$ and
\begin{equation}\label{holder_contin}
\Big[{\bf m}\Big]_{C^{\alpha_0}(P_{\frac32}(0,4))}\le C(\lambda) \big\|\nabla{\bf m}\big\|_{M^{2,2}(P_2(0,4))}.
\end{equation}
By higher order regularity theorem of LLG (\ref{1.1}) and (\ref{holder_contin}), we conclude that ${\bf m}\in C^\infty(P_\frac32(0,4))$
and satisfies the estimate (\ref{e-conclusion}). This completes the proof.
\end{proof}

Now we return to the proof of Theorem \ref{th:1.1}. By Lemma \ref{e-regularity} and the estimate (\ref{x-estimate7}), we want to show
\\

\noindent{\it Claim}. $T^k=+\infty$, and $\displaystyle \sup_{k\ge 1} \big\|{\bf m}^k\big\|_{C^l(\mathbb R^n\times [\delta,+\infty))}$ is uniformly bounded
 for all $l\ge 1$ and $\delta>0$.

For, otherwise, $T^k<+\infty$.  Then we must have that
$$\lim_{t\uparrow T_k}\big\|\nabla{\bf m}^k(t)\big\|_{L^\infty(\mathbb R^n)}=+\infty.$$
On the other hand, by (\ref{x-estimate7}),  we have
\begin{equation}\label{e-condition2}
\sup_{t\in [\frac{T_k}2, T_k]}\big\|\nabla{\bf m}^k\big\|_{M^{2,2}(\mathbb R^n)}\le C\epsilon_0.
\end{equation}
This clearly implies that for any $x_0\in \mathbb R^n$,
\begin{equation}\label{e-condition3}
\Big\|\nabla{\bf m}^k\Big\|_{M^{2,2}\big(P_{\frac{\sqrt{T_k}}2}(x_0, T_k)\big)}\le C\epsilon_0.
\end{equation}
Since ${\bf m}^k\in C^0([0,T^k], H^\infty_*(\mathbb R^n))$, $\nabla^2{\bf m}^k \in L^2_{\rm{loc}}(\mathbb R^n\times [\frac{T_k}2, T_k])$.
Thus we  can apply Lemma \ref{e-regularity} to conclude that for any $k$,
$\displaystyle\|\nabla{\bf m}^k(t)\|_{L^\infty(\mathbb R^n)}$ remains to be bounded, as $t$ tends to $T_k$.
We reach a contradiction. Thus $T_k=+\infty$. Now by applying Lemma \ref{e-regularity} again, we would conclude
that ${\bf m}^k\in C^\infty(\mathbb R^n\times (0, +\infty))$. Moreover, for any $\delta>0$,
\begin{equation}\label{uniform_bound}
\sup_{k\ge 1}\big\|{\bf m}^k\big\|_{C^l\big(\mathbb R^n\times [\delta,+\infty)\big)}\le C(l, \delta,\epsilon_0).
\end{equation}
This proves the claim.

After taking possible subsequences, we may assume that there exists a ${\bf m}\in C^\infty(\mathbb R^n\times (0,+\infty),\mathbb S^2)$
such that ${\bf m}^k$ converges to ${\bf m}$ in $C^l_{\rm{loc}}(\mathbb R^n\times (0,+\infty))$. It is not hard to see that
${\bf m}|_{t=0}={\bf m}_0$ in the sense of trace. By the lower semicontinuity and (\ref{x-estimate7}), we also have
$$\sup_{t\ge 0} \big\|\nabla{\bf m}\big\|_{M^{2,2}(\mathbb R^n)}\le \big\|u({\bf m})\big\|_{X_\infty^p}\le\liminf_{k\rightarrow\infty}\big\|u({\bf m}^k)\big\|_{X_\infty^p}
\le C\big\|\nabla {\bf m}_0\big\|_{M^{2,2}(\mathbb R^n)}\le C\epsilon_0.$$
Moreover, by the $\epsilon$-regularity Lemma \ref{e-regularity}, we conclude that ${\bf m}$ also satisfies the derivative estimates (\ref{smoothness}).

Finally, if ${\bf m}_0\in H^1_*(\mathbb R^n,\mathbb S^2)$, then we have
$$\lim_{k\rightarrow \infty}\int_{\mathbb R^n}|\nabla({\bf m}_0^k-{\bf m}_0)|^2\,dx=0.$$
Since ${\bf m}^k\in C^\infty(\mathbb R^n\times (0,+\infty))$ satisfies the equation (\ref{LLG}) in $\mathbb R^n\times (0,+\infty)$, we can multiply
(\ref{LLG}) by $\partial_t {\bf m}^k$ and integrate the resulting equation on $\mathbb R^n$ to get the global energy equality:
\begin{equation}\label{energy_ineq1}
\frac{d}{dt}\int_{\mathbb R^n}|\nabla {\bf m}^k(t)|^2\,dx
=-\frac{\lambda}{1+\lambda^2}\int_{\mathbb R^n}|\partial_t{\bf m}^k(t)|^2\,dx.
\end{equation}
Integrating (\ref{energy_ineq1}) over $[0,t]$ yields
\begin{equation}\label{energy_ineq2}
\int_{\mathbb R^n}|\nabla {\bf m}^k(t)|^2\,dx
+\frac{\lambda}{1+\lambda^2}\int_0^t\int_{\mathbb R^n}|\partial_t{\bf m}^k|^2\,dxdt
=\int_{\mathbb R^n}|\nabla {\bf m}^k_0|^2\,dx.
\end{equation}
Sending $k$ to $\infty$ in (\ref{energy_ineq2}) and applying the lower semicontinuity, we conclude that ${\bf m}$ satisfies the energy inequality (\ref{energy_ineq0}).

Thus the proof of Theorem \ref{th:1.1} would be complete, if we have shown the uniqueness of ${\bf m}$ in its class when $\nabla{\bf m}_0\in L^2(\mathbb R^n)$.
This part follows from the lemma below.

\subsection{Uniqueness of LLG (\ref{1.1})}  Motivated by the uniqueness theorem on the heat flow of harmonic maps  established
by Huang-Wang \cite{Huang-Wang1}, we will establish that LLG (\ref{1.1}) enjoys the following uniqueness property for initial data in $M^{2,2}(\mathbb R^n)$.
\begin{lemma}\label{uniqueness} There exist $\epsilon_0>0$ and $c_0>0$ depending on $n,\lambda$ such that for ${\bf m}_0\in H^1_*(\mathbb R^n, \mathbb S^2)$
and  $i=1,2$, if ${\bf m}^{(i)}:\mathbb R^n\times [0,+\infty)\to \mathbb S^2$ are  weak solutions of LLG (\ref{1.1}), under the initial data ${\bf m}_0$,
satisfying\\
(i) ${\bf m}^{(i)}\in C^0([0,+\infty), H^1_*(\mathbb R^n,\mathbb S^2))$, $\partial_t {\bf m}^{(i)}\in L^2([0,+\infty), L^2(\mathbb R^n))$,
and $\nabla^2{\bf m}^{(i)}\in L^2_{\rm{loc}}(\mathbb R^n\times [\delta,+\infty))$ for $\delta>0$; and \\
(ii)
\begin{equation}\label{small_morrey2}
\sup_{t\ge 0} \big\|\nabla{\bf m}^{(i)}(t)\big\|_{M^{2,2}(\mathbb R^n)}\le c_0\epsilon_0,
\end{equation}
then ${\bf m}^{(1)}\equiv {\bf m}^{(2)}$ in $\mathbb R^n\times [0,+\infty)$.
\end{lemma}
\begin{proof} The idea is originally due to Huang-Wang \cite{Huang-Wang1} on the heat flow of harmonic maps. Here we adapt it to LLG (\ref{1.1}).
Applying Lemma \ref{e-regularity} to ${\bf m}^{(i)}$, we see that ${\bf m}^{(i)}\in C^\infty(\mathbb R^n\times (0,+\infty))$ and
\begin{equation}\label{grad_est}
t^{\frac12}\big\|\nabla{\bf m}^{(i)}(t)\big\|_{L^\infty(\mathbb R^n)} \le \sup_{\tau\ge 0} \big\|\nabla{\bf m}^{(i)}(\tau)\big\|_{M^{2,2}(\mathbb R^n)}\le c_0\epsilon_0,
\ \forall t>0.
\end{equation}
Set ${\bf m}={\bf m}^{(1)}-{\bf m}^{(2)}$. Then we have
\begin{equation}\label{LLG1}
\partial_t {\bf m}=\lambda \Delta{\bf m}+\lambda\big (|\nabla {\bf m}^{(1)}|^2 {\bf m}^{(1)}-|\nabla {\bf m}^{(2)}|^2 {\bf m}^{(2)}\big)
-\big({\bf m}\times \Delta {\bf m}^{(1)}+{\bf m}^{(2)}\times \Delta{\bf m}\big).
\end{equation}
Since $\Delta {\bf m}^{(i)}(t), \Delta{\bf m}(t)\in L^2_{\rm{loc}}(\mathbb R^n)$ for a.e. $t>0$ and ${\bf m}\in L^2(\mathbb R^n)$, we can multiply
(\ref{LLG1}) by ${\bf m}$ and integrate the resulting equation over $\mathbb R^n$  to obtain
\begin{eqnarray}
\frac{d}{dt}\int_{\mathbb R^n}|{\bf m}|^2+2\lambda \int_{\mathbb R^n}|\nabla{\bf m}|^2&=&2\lambda
\int_{\mathbb R^n} \big (|\nabla {\bf m}^{(1)}|^2 {\bf m}^{(1)}-|\nabla {\bf m}^{(2)}|^2 {\bf m}^{(2)}\big)\cdot{\bf m}
+2\int_{\mathbb R^n} \big(\nabla {\bf m}^{(2)}\times \nabla {\bf m}\big)\cdot{\bf m}\nonumber\\
&=&I+II. \label{energy_est2}
\end{eqnarray}
Observe that
$${\bf m}^1\cdot {\bf m}=\frac12 |{\bf m}|^2, \  {\bf m}^2\cdot {\bf m}=-\frac12 |{\bf m}|^2.$$
This and (\ref{grad_est}) imply
$$I=\lambda
\int_{\mathbb R^n} \big (|\nabla {\bf m}^{(1)}|^2 +|\nabla {\bf m}^{(2)}|^2\big)|{\bf m}|^2
\le \frac{C\epsilon_0^2}{t}\int_{\mathbb R^n} |{\bf m}|^2.
$$
By H\"older's inequality and (\ref{grad_est}), we have
$$|II|\le \lambda\int_{\mathbb R^n}|\nabla{\bf m}|^2+\lambda^{-1}\int_{\mathbb R^n}|\nabla{\bf m}^{(2)}|^2|{\bf m}|^2
\le \lambda\int_{\mathbb R^n}|\nabla{\bf m}|^2+\frac{C\epsilon_0^2}{t}\int_{\mathbb R^n} |{\bf m}|^2.
$$
Putting these estimates into (\ref{energy_est2}) yields
$$
\frac{d}{dt}\int_{\mathbb R^n}|{\bf m}(t)|^2\le \frac{C\epsilon_0^2}{t}\int_{\mathbb R^n} |{\bf m}(t)|^2.
$$
Hence we have
$$
\frac{d}{dt}\Big(t^{-\frac12}\int_{\mathbb R^n}|{\bf m}(t)|^2\Big)
\le \frac{C\epsilon_0^2-\frac12}{t^\frac32}\int_{\mathbb R^n} |{\bf m}(t)|^2\le 0.
$$
This yields
$$t^{-\frac12}\int_{\mathbb R^n}|{\bf m}(t)|^2\le \lim_{s\downarrow 0} s^{-\frac12}\int_{\mathbb R^n}|{\bf m}(s)|^2.$$
On the other hand, since ${\bf m}(0)=0$, we have that $\displaystyle {\bf m}(x,\tau)=\int_0^\tau \partial_t {\bf m}(x,t)\,dt, \ {\rm{a.e.}} \ x\in\mathbb R^n, \tau>0$.
By H\"older's inequality, we have
$$s^{-\frac12}\int_{\mathbb R^n}|{\bf m}(x,s)|^2\,dx\le s^\frac12\int_0^s\int_{\mathbb R^n}|\partial_t{\bf m}(x,t)|^2\,dxdt\le Cs^\frac12\rightarrow
0, \ {\rm{as}}\ s\downarrow 0.$$
Therefore ${\bf m}(x,t)=0$ for a.e. $(x,t)\in \mathbb R^n\times [0,+\infty)$. This completes the proof.
\end{proof}
It is clear that the global solution ${\bf m}$ obtained in Theorem \ref{th:1.1} satisfies all the conditions in Lemma \ref{uniqueness}.
Hence ${\bf m}$ is unique in its own class. The proof of Theorem \ref{th:1.1} is now complete.

\subsection{Stability of the covariant LLG(\ref{3.1}) and the covariant GL(\ref{3.2})}
In this final section, we will show that the covariant LLG equation (\ref{3.1}) and (\ref{3.2}), under the Coulomb gauge, enjoy the following
stability property in the space $X_\infty^p$. This may have its own interest.

\begin{lemma}\label{stable} There exist $\epsilon_0>0$ and $c_0, c_1>0$ depending only on $n, \lambda, p$ with the following property:
For $i=1,2$, if $v_0^{(i)}\in M^{2,2}(\mathbb R^n)$ satisfies
\begin{equation}\label{e-condition10}
\big\|v_0^{(i)}\big\|_{M^{2,2}(\mathbb R^n)}\le\epsilon_0,
\end{equation}
and ${\bf u}^{(i)}=(u_0^{(i)}, u^{(i)}):\mathbb R^n\times [0,+\infty)\to \mathbb C^{n+1}$ is a solution of (\ref{3.1}) and (\ref{3.2})
given by the Duhamel formula (\ref{duhamel1}), 
with initial data $v^{(i)}_0$ and the Coulomb gauge $a^{(i)}$, that satisfies
\begin{equation}\label{x-estimate26}
\big\|u^{(i)}\big\|_{X_\infty^p}\le c_0\epsilon_0,
\end{equation}
for some $p\in (3, \frac{10}3)$.
Then 
\begin{equation}\label{x-stability}
\Big\|u^{(1)}-u^{(2)}\Big\|_{X_\infty^p}\le c_1\Big\|v_0^{(1)}-v_0^{(2)}\Big\|_{M^{2,2}(\mathbb R^n)}.
\end{equation}
\end{lemma}
\begin{proof} To emphasize the dependence of the Coulomb gauge $a^{(i)}$ on $u^{(i)}$, set $a(u^{i})=a^{(i)}$  as the Coulomb gauge
associated with $u^{(i)}$. Also set $a(u)=a(u^{(1)})-a(u^{(2)})$, $a_0^{(1)}(u)=a_0^{(1)}(u^{(1)})-a_0^{(1)}(u^{(2)})$, $a_0^{(2)}(u)=a_0^{(2)}(u^{(1)})-a_0^{(2)}(u^{(2)})$,
$u=u^{(1)}-u^{(2)}$, and $v_0=v_0^{(1)}-v_0^{(2)}$.
By Duhamel's formula (\ref{duhamel1}), we have
\begin{equation}\label{duhamel6}
u(t)=S(t) v_0+\sum_{l=1}^3 S*\big(f^{(l)}(u^{(1)})-f^{(l)}(u^{(2)})\big)(t).
\end{equation}
It is clear that Lemma 3.2 yields
\begin{equation}\label{x-estimate20}
\Big\|S(t)v_0\Big\|_{X_\infty^p}\le c_1\big\|v_0\big\|_{M^{2,2}(\mathbb R^n)}.
\end{equation}
Observe that
$$\Big|f^{(1)}(u^{(1)})-f^{(1)}(u^{(2)})\Big|\lesssim \big(|u^{(1)}|^2+|u^{(2)}|^2\big)|u|,$$
so that we can apply Lemma 4.1, Lemma 4.2, and Lemma 4.3, and (\ref{x-estimate26})  to get
\begin{equation}\label{x-estimate10}
\Big\|S*\big(f^{(l)}(u^{(1)})-f^{(l)}(u^{(2)})\big)\Big\|_{X_\infty^p}
\lesssim \Big(\big\|u^{(1)}\big\|_{X_\infty^p}^2+\big\|u^{(2)}\big\|_{X_\infty^p}^2\Big)\big\|u\big\|_{X^p_\infty}
\lesssim \epsilon_0^2 \big\|u\big\|_{X^p_\infty}.
\end{equation}
Observe
$$\Big|f^{(2}(u^{(1)})-f^{(2)}(u^{(2)})\Big|\lesssim |a(u^{(1)})||\nabla u|+|a(u)||\nabla u^{(2)}|+|a_0^{(1)}(u^{(1)})|u|+|a_0^{(1)}(u)||u^{(2)}|.$$
Applying Lemma 3.2, we find that
$$a(u^{(i)}), a(u)\in M^{\frac{2p}{4-p},2}(\mathbb R^n),\  a_0^{(1)}(u^{(1)}), a_0^{(1)}(u)\in
M^{p,2}(\mathbb R^n), \ {\rm{and}} \ a_0^{(2)}(u^{(1)}), a_0^{(2)}(u)\in
M^{\frac{p}{4-p},2}(\mathbb R^n),$$
along with the estimates:
$$\big\|a(u^{(i)})\big\|_{M^{\frac{2p}{4-p},2}(\mathbb R^n)}\lesssim \big\|u^{(i)}\big\|_{M^{p,2}(\mathbb R^n)}^2, \ i=1,2, $$
$$\big\|a(u)\big\|_{M^{\frac{2p}{4-p},2}(\mathbb R^n)}\lesssim
\Big( \big\|u^{(1)}\big\|_{M^{p,2}(\mathbb R^n)}+\big\|u^{(2)}\big\|_{M^{p,2}(\mathbb R^n)}\Big)\big\|u\big\|_{M^{p,2}(\mathbb R^n)},$$
$$\Big\|a_0^{(1)}(u^{(1)})\Big\|_{M^{p,2}(\mathbb R^n)}\lesssim \big\|u^{(1)}\big\|_{M^{p,2}(\mathbb R^n)}  \big\|\nabla u^{(1)}\big\|_{M^{2,2}(\mathbb R^n)},$$
$$\Big\|a_0^{(1)}(u)\Big\|_{M^{p,2}(\mathbb R^n)}\lesssim
\Big(\big\|u\big\|_{M^{p,2}(\mathbb R^n)}  \big\|\nabla u^{(1)}\big\|_{M^{2,2}(\mathbb R^n)}+
\big\|u^{(2)}\big\|_{M^{p,2}(\mathbb R^n)}  \big\|\nabla u\big\|_{M^{2,2}(\mathbb R^n)}\Big),$$
$$\Big\|a_0^{(2)}(u^{(1)})\Big\|_{M^{\frac{p}{4-p},2}(\mathbb R^n)}
\lesssim \big\|u^{(1)}\big\|_{M^{p,2}(\mathbb R^n)}^4,$$
and
$$
\Big\|a_0^{(2)}(u)\Big\|_{M^{\frac{p}{4-p},2}(\mathbb R^n)}
\lesssim \Big(\big\|u^{(1)}\big\|_{M^{p,2}(\mathbb R^n)}^3+\big\|u^{(2)}\big\|_{M^{p,2}(\mathbb R^n)}^3\Big)
\big\|u\big\|_{M^{p,2}(\mathbb R^n)}.
$$
Thus by Lemma 4.4, Lemma 4.5, and Lemma 4.6 we have
\begin{equation}\label{x-estimate11}
\Big\|S*\big(f^{(2}(u^{(1)})-f^{(2)}(u^{(2)})\big)\Big\|_{X_\infty^p}\lesssim \Big(\big\|u^{(1)}\big\|_{X_\infty^p}^2+\big\|u^{(2)}\big\|_{X_\infty^p}^2\Big)\big\|u\big\|_{X^p_\infty}
\le C\epsilon_0^2 \big\|u\big\|_{X^p_\infty}.
\end{equation}
Observe
$$\Big|f^{(3}(u^{(1)})-f^{(3)}(u^{(2)})\Big|\lesssim \big|a(u^{(1)})\big|^2|u|+|a(u)|\big[|a(u^{(1)})|+|a(u^{(2)})|\big]|u^{(2)}|+\big|a_0^{(2)}(u^{(1)})\big||u|
+\big|a_0^{(2)}(u)\big||u^{(2)}|.$$
Thus by Lemma 4.7, Lemma 4.8, and Lemma 4.9 we have
\begin{equation}
\label{x-estimate12}
\Big\|S*\big(f^{(3}(u^{(1)})-f^{(3)}(u^{(2)})\big)\Big\|_{X_\infty^p}
\lesssim \Big(\big\|u^{(1)}\big\|_{X_\infty^p}^4+\big\|u^{(2)}\big\|_{X_\infty^p}^4\Big)
\big\|u\big\|_{X_\infty^p}\le C\epsilon_0^4\big\|u\big\|_{X^p_\infty}.
\end{equation}
Substituting (\ref{x-estimate10}), (\ref{x-estimate11}), and (\ref{x-estimate12}) into (\ref{duhamel6}), we arrive at
\begin{equation}\label{x-estimate13}
\big\|u\big\|_{X_\infty^p}\le \frac{c_1}2\big\|v_0\big\|_{M^{2,2}(\mathbb R^n)}+C\big(\epsilon^2_0+\epsilon_0^4\big)\big\|u\big\|_{X_\infty^p}.
\end{equation}
This implies (\ref{x-stability}), provided that $\epsilon_0>0$ is chosen to be sufficiently small.
The proof is complete.
\end{proof}

\bigskip
\noindent{\bf Acknowledgments}. J. Lin is partially supported by NSF of China (Grant 11001085) and the Ph.D. Programs Foundation of Minstry of Education of China (Grant 20100172120026)
B. Lai  is partially supported by NSF of China (Grants 11201119 and 11126155).
C. Wang  is partially supported by  NSF grants 1001115 and 1265574, NSF of China grant
11128102, and a Simons Fellowship in Mathematics. The work was completed while both J. Lin and B. Lai visited Department of Mathematics,
University of Kentucky. They would like to thank the Department for its hospitality and excellent research environment. 


\end{document}